\documentclass[11pt, reqno, letterpaper]{amsart}

\usepackage{esint}
\usepackage{amssymb}

\usepackage[
hypertexnames=false, colorlinks, 
linkcolor=black,citecolor=black,urlcolor=black, linktocpage=true
]%
{hyperref} 
\hypersetup{bookmarksdepth=3}

	\normalbaselines						  %

\usepackage{graphicx}
\usepackage{
mathrsfs}

\newtheorem{thm}{Theorem}[section]
\newtheorem{lm}[thm]{Lemma}

\newtheorem{prop}[thm]{Proposition}

\theoremstyle{definition}
\newtheorem{df}[thm]{Definition}
\newtheorem*{df*}{Definition}

\theoremstyle{remark}

\newtheorem*{rem*}{Remark}

\numberwithin{equation}{section}

\newcommand{\ci}[1]{_{ {}_{\scriptstyle #1}}}
\newcommand{\ti}[1]{_{\scriptstyle \text{\rm #1}}}

\newcommand{\ut}[1]{^{\scriptstyle \text{\rm #1}}}

\newcommand{\cA}{\mathcal{A}}
\newcommand{\cB}{\mathcal{B}}
\newcommand{\cD}{\mathscr{D}}

\newcommand{\cX}{\mathcal{X}}

\newcommand{\cM}{\mathcal{M}}

\newcommand{\cG}{\mathcal{G}}

\newcommand{\cJ}{\mathcal{J}}
\newcommand{\cK}{\mathcal{K}}
\newcommand{\cH}{\mathcal{H}}

\newcommand{\fS}{\mathfrak{S}}
\newcommand{\f}{\varphi}

\newcommand{\e}{\varepsilon}

\newcommand{\A}{\mathbb{A}}

\newcommand{\R}{\mathbb{R}}
\newcommand{\Z}{\mathbb{Z}}

\newcommand{\E}{\mathbb{E}}

\newcommand{\ch}{\operatorname{ch}}

\newcommand{\1}{\mathbf{1}}

\newcommand{\wt}{\widetilde}

\newcommand{\clos}{\operatorname{clos}}

\newcommand{\cz}{Calder\'{o}n--Zygmund\ }
\newcommand{\rk}{\operatorname{rk}}

\newcommand{\spn}{\operatorname{span}}

\newcommand{\La}{\langle }
\newcommand{\Ra}{\rangle }

\newcommand{\ssd}{{\scriptscriptstyle\Delta}}
\newcommand{\sd}{{\scriptstyle\Delta}}

\newcommand{\fdot}{\,\cdot\,}

\def\cyr{\fontencoding{OT2}\fontfamily{wncyr}\selectfont}
\DeclareTextFontCommand{\textcyr}{\cyr}

{\end{list}}      

\renewcommand{\labelenumi}{(\roman{enumi})}

\newcounter{vremennyj}

\newcommand\cond[1]{\setcounter{vremennyj}{\theenumi}\setcounter{enumi}{#1}\labelenumi\setcounter{enumi}{\thevremennyj}}

\begin{document}
\title[Weighted martingale multipliers]%
{%
Weighted martingale multipliers in non-homogeneous setting and outer measure spaces 
}
\author[C.~Thiele]{Christoph Thiele}
\thanks{CT is supported by the NSF grant DMS-1001535}
\address{Mathematisches Institut,  	Universit\"at Bonn 
 \newline\indent	Endenicher Allee 60,	D - 53115 Bonn, Germany
 \newline  \indent
 {and} \vskip1mm  
 Department of Mathematics, UCLA, Los Angeles, CA 90095-1555
  }
 \email{thiele@math.uni-bonn.de \textrm{(C.\ Thiele)}}
\author[S.~Treil]{Sergei Treil}
\thanks{ST is partially supported by the NSF grant DMS-1301579}
\address{Department of Mathematics, Brown University, 151 Thayer St., Box 1917, \newline\indent Providence, RI 02912, USA}
\email{treil@math.brown.edu \textrm{(S.\ Treil)}}
\author[A.~Volberg]{Alexander Volberg}
\thanks{AV is partially supported by the NSF grant DMS-1265549 and by the Hausdorff Institute for Mathematics, Bonn, Germany}
\address{Department of Mathematics, Michigan Sate University, East Lansing, MI. 48823}
\email{volberg@math.msu.edu \textrm{(A.\ Volberg)}}
\makeatletter
\@namedef{subjclassname@2010}{
  \textup{2010} Mathematics Subject Classification}
\makeatother
%
%
\begin{abstract}
We investigate the unconditional  basis property  of martingale differences in weighted $L^2$ spaces in the non-homogeneous situation (i.e.~when the reference measure is not doubling). 
 Specifically, we prove that finiteness of  the quantity $[w]\ci{A_2}=\sup_I \,\La w\Ra_I\La w^{-1}\Ra_I$, defined through 
averages $\La\fdot \Ra_I$ relative to the reference measure $\nu$, implies that each  martingale 
transform relative to $\nu$ is bounded in $L^2(w\, d\nu)$. Moreover, we prove  the linear in $[w]\ci{A_2}$ estimate of the unconditional basis constant of the Haar system. 

Even in the classical case of the standard dyadic lattice in $\R^n$, where the results about unconditional basis and linear in $[w]\ci{A_2}$ estimates are known, our result gives something new, because all the estimates are independent of the dimension $n$.

Our approach combines the technique of outer measure spaces with the Bellman function argument.

\end{abstract}
\maketitle

\section{Introduction}
The classical Haar system (which is an orthonormal basis in $L^2(\R)$) is an unconditional basis in the weighted space $L^2(w)= L^2(\R, w)$ if and only if the weight $w$ satisfies the so-called dyadic Muckenhoupt $A_2$ condition. 

This result can be easily generalized to the bases of Haar subspaces (a.k.a~martingale difference spaces) in $L^2(\R^d)$, where for a cube $Q$ the corresponding Haar subspace $H\ci Q$ is the subspace of dimension $2^d-1$, consisting of functions supported on $Q$, constant on children of $Q$ and orthogonal (in the unweighted $L^2$) to constants.  

These results we proved in \cite{GundyWheeden1973}, where among other problems the weighted Littlewood--Paley estimates were considered: the equivalence if these estimates to the unconditional basis property is just the standard fact of the theory of bases. The fact that the Muckenhoupt condition is equivalent to the Haar system being a basis (which from the modern standpoint is almost trivial) was established earlier in \cite{Krantzberg-1971}. 

In this paper we investigate what happens in the non-homogeneous situation. A typical example is the standard dyadic lattice in $\R^d$, with the underlying measure being not the Lebesgue measure, but an arbitrary Radon measure $\nu$ in $\R^n$. Then our Haar subspaces are orthogonal to constants in $L^2(\nu)$ and we want to describe weighs $w$ for which the Haar subspaces form an unconditional basis in $L^2(wd\nu)$. We, in fact, consider a more general martingale situation, when we do not have any bound on the dimension of the Haar subspaces, and moreover some (or all) Haar subspaces can be infinite-dimensional.  

We prove that in this general case the corresponding martingale $A_2$ condition is also necessary and sufficient for the system of Haar subspaces to be an unconditional basis in the weighted space $L^2(wd\nu)$.  We also prove that, as in the homogeneous case, the \emph{unconditional basis constant} of this system  admits the estimate which is linear in the $A_2$ characteristic $[w]\ci{A_2}$ of the weight $w$. 

The problem, as we explain below in Section \ref{s:Setup} is equivalent to the weighted estimates of the so-called \emph{martingale multipliers}. Thus, the martingale multipliers are the most natural ``singular'' martingale transforms.

When we started the project, we were not sure that the $A_2$ condition is sufficient for the Haar system being an unconditional basis (necessity is a simple fact), and the linear in the $A_2$ characteristic estimate seemed like a long shot. 

The reason for such pessimism was that the the non-homogeneous situation is very different from the classical dyadic case (or from a homogeneous situation). For example in the classical dyadic situation the dyadic  $A_2$ condition is sufficient for the boundedness of general Haar (martingale) transforms. Here a Haar (martingale) transform $T$ is a bounded (in the unweighted $L^2$) operator which is diagonal in the orthogonal basis of the Haar subspaces. And it is known and is not hard to show that in the classical dyadic situation  if the weight satisfy the dyadic $A_2$ condition, then  all Haar transforms $T$ are bounded in $L^2(w)$ and 
\[
\|T\|\ci{L^2(w)} \le C([w]\ci{A_2}) \|T\|\ci{L^2}. 
\]
Moreover, it was shown that in the classical situation the estimate is linear in $[w]\ci{A_2}$. 

The non-homogeneous situation is quite different. An example (which we present in Section \ref{counterE} below) shows that in the non-homogeneous situation  Haar (martingale) transforms (not martingale multiplier!) are not necessarily bounded in $L^2(w)$ with $w$ satisfying the $A_2$ condition, even in the simple model case of the standard dyadic lattice in $\R^2$ with a general underlying measure $\nu$.%
\footnote{The case of dyadic lattice in $\R$, with a general underlying measure $\nu$ is an exception: all Haar subspaces are one-dimensional, so any Haar transform is a \emph{Haar multiplier}  (all blocks are multiples of identity), and it is the main result of this paper that the bounded in the unweighted $L^2$ Haar multipliers act in $L^2(w)$ if $w$ satisfies the $A_2$ condition.  } 

Here by \emph{martingale transform} we mean an operator which is diagonal in the basis of the martingale difference spaces; but \emph{martingale multiplier} is a martingale transform with all blocks being just multiples of identity.  

Interesting examples illustrating intricacies of the non-homogeneous case can be found in \cite{LMP}, \cite{Tr_Comm-para2010}. For example,  in \cite{Tr_Comm-para2010} an unbounded (in unweighted non-homogeneous $L^p$, $1<p<\infty$, $p\ne 2$) martingale transform   with uniformly bounded diagonal blocks was constructed there. Note, that the above mentioned counterexample  in the weighted $L^2$ is obtained by constructing a weight and blocks that are uniformly bounded in the unweighted case, but fail to be uniformly bounded in the weighted $L^2$. 

Another interesting result from \cite{Tr_Comm-para2010} that the boundedness of a paraproduct in the unweighted $L^2$ is not equivalent to the symbol being in the martingale BMO; symbol in BMO is only a sufficient condition. While, like in the homogeneous case, it is sufficient to test the boundedness of a paraproduct on characteristic functions of ``intervals'', this testing condition in the non-homogeneous case depends on $p$ and is not equivalent to the symbol being in BMO. 

Also, interesting results about non-homogeneous case were obtained in  \cite{LMP}, where a dyadic analogue of the Hilbert Transform was considered.



As for the weighted estimates for martingale transform, it was shown recently (as a byproduct of other results) in \cite{TV} that the $A_2$ condition is sufficient for the boundedness of the so-called $L^1\otimes L^1$ normalized martingale transforms (although we  did not get the linear in $A_2$ characteristic estimate in  \cite{TV} ). The $L^1\otimes L^1$normalization condition means that each diagonal block $T\ci I$ is represented an integral operator with kernel $K\ci I$, $\| K\ci I\|_\infty \le \nu(I)^{-1}$. Note that the martingale multipliers considered in this paper are generally not $L^1\otimes L^1$ normalized; they are  $L^1\otimes L^1$ normalized only in the homogeneous case when the ratio of the measures of  a parent and a child is uniformly bounded.   
 
In \cite{TV} the authors asked specifically about the conditions on the boundedness of the martingale multipliers in the non-homogeneous case. We answer this question here.

Finally, we suspect that $L^1\otimes L^1$ normalized martingale transforms are another natural class of martingale  operators. In the classical homogeneous situation  a \cz operator can be represented as a weighted average of $L^1\otimes L^1$ normalized Haar shifts with weights decaying exponentially in complexity, and, in fact, we can prove that a similar decomposition holds in the non-homogeneous case. And a Haar shift of complexity $n$ is essentially the sum of $n$ martingale transforms, each in its own filtration obtained from the original one by skipping generations. This will be presented in a subsequent paper resulting in the proof of $A_2$ conjecture for arbitrary reference measure.


The authors are grateful to Carlos P\'erez for organizing an inspiring Summer School
in Santander, Spain, where the research on this paper was initiated.

\section{Setup}
\label{s:Setup}
Consider a $\sigma$-finite measure space $(\cX, \fS, \nu)$ with the filtration (i.e.~with the sequence of increasing $\sigma$-algebras) $\fS_n$, $n\in\Z$, $\fS_n\subset\fS_{n+1}$. 

We assume that each $\sigma$-algebra $\fS_n$ is \emph{atomic}, meaning that 
there exists a countable \emph{disjoint} collection $\cD_n$ of the sets of positive measure (atoms), such that every $A\in\fS_n$ is a union of sets $I\in\cD_n$. 

The fact that $ \fS_n\subset \fS_{n+1}$ means that every $I\in \cD_n$ is at most countable union of $I'\in\cD_{n+1}$. 

We denote by $\cD=\bigcup_{n\in\Z}\cD_n$ the collection of all atoms  (in all generations). 

The typical example will be the filtration given by a dyadic lattice in $\R^d$, so the notation $\cD$. Note, that we do not assume any homogeneity  in our setup, so the more interesting example will be the same dyadic lattice in $\R^d$, but the underlying measure is an arbitrary Radon measure $\nu$.  

We will allow a situation  when an atom $I$ belongs to several (even infinitely many) generations $\cD_n$. The leading particular case is an arbitrary measure $\nu$ on $\R^d$, whose support is called $\cX$, and the usual standard dyadic cube $I$ then will be included in the filtration only if $\nu(I)>0$ (in particular, cubes $I$ such that $I\cap\cX=\emptyset$ are not in $\cD$). In this example a cube $I$ of the filtration can have only one child  in the next level of the filtration (and not $2^d$ children), and this can go on for some time.
However, we will not allow $I$ to be in all generations, because in this case nothing interesting happens on the interval $I$. 

We  usually will not assign a special symbol for the underlying measure of the set, meaning that we use $|A|$ instead of $\nu(A)$. But the reader is reminded that this is just a notational convention, our  ``$|I|$" can be very far from being Lebesgue measures of $I$. The notation is chosen to emphasize that ``everything is as if we would have Lebesgue measure". However, the reader should remember that we are in the ``mine field": if we would change our operator slightly we would be in trouble. So one should be quite careful, and the proofs hide many surprises.

Note that our filtered space $\cX$ can be represented as a countable (finite or infinite) direct sum of the filtered spaces treated in \cite{Tr_Comm-para2010}, so all the results from \cite{Tr_Comm-para2010} hold in our case.  

For an interval $I\in\cD$ define it lower and upper ranks $\rk_-(I)$ and $\rk_+(I)$ as 
\[
\rk_-(I) = \inf\{n: I\in\cD_n\}, \qquad \rk_+(I) :=\sup\{ n: I\in \cD_n\}. 
\] 

\begin{df}
\label{df:child} Let $I\in\cD$, and let $n=\rk_+(I)$. Then the intervals $I'\in\cD_{n+1}$ such that $I'\subset I$ are called the \emph{children} of $I$. The collection of all children of $I$ will be denoted by $\ch(I)$. 

If $\rk_+(I)=+\infty$  we set $\ch(I)=\{I\}$.

We will also need the notion of \emph{martingale} or \emph{time} children of an interval (atom) $I\in\cD_n$, 
\[
\ch\ut{t} (I,n) =\{I'\in \cD_{n+1} : I'\subset I\}. 
\]
Note that if $n<\rk_+(I)$ then $\ch\ut{t} (I,n) =\{I\}$. 

Note also that in the last definition we require that $I\in\cD_n$. Since $I$ can be in several  $\cD_k$s, we define the martingale children for the pair $(I,n)$. 
%
\end{df}

\subsubsection{Martingale differences, martingale difference spaces}
For a measurable $I$ we define the average (recall that $|I|$ is a short hand for $\nu(I)$)
\[
\La f\Ra\ci I :=|I|^{-1} \int_I f d\nu, 
\]
and the averaging operator $\E\ci I$ by
\[
\E\ci I f = \La f\Ra\ci I \1\ci I. 
\]
For $I\in\cD$ the \emph{martingale difference operator} $\Delta\ci I$ is given by
\begin{align}
\label{Delta_I}
\Delta\ci I f := \sum_{I'\in\ch(I) } \E\ci{I'} f \  - \ \E\ci I f ; 
\end{align}
note that formally 
if $\ch(I)=\{I\}$ then 
 $\Delta\ci I =0$. 
 
 For $I\in\cD_n$ the \emph{time} martingale difference $\Delta\ci{(I,n)}\ut{t}$ is given by 
\[
\Delta\ci{(I,n)}\ut{t} := \sum_{I'\in\ch\ut{t}(I,n) } \E\ci{I'} f \  - \ \E\ci I f
\]

Let $\E_n$ be the conditional expectation with respect to the $\sigma$-algebra  $\fS_n$
\[
\E_n:= \sum_{I\in\cD_n} \E\ci I, 
\]
and let $\Delta_n$ be the corresponding martingale difference
\[
\Delta_n := \E_{n} -\E_{n-1} = \sum_{I\in\cD_{n-1}} \Delta\ci{(I, n-1)}\ut{t} = \sum_{I\in\cD: \,\rk_+(I)=n-1} \Delta\ci{I}
\]

Let $\fS_\infty$ be the $\sigma$-algebra generated by $\fS_n$, $n\in\Z$, and let $\fS_{-\infty}:= \bigcap_{n\in\Z} \fS_n$. To simplify the notation we assume that $\fS_\infty=\fS$, so we deal only with $\fS_\infty$ measurable functions. 

The sigma algebra $\fS_{-\infty}$ is generated by the collection $\cD_{-\infty}$ of atoms $I$ of form $I=\bigcup_{k} I_k$ where $I_k\in\cD_k$, $I_k\subset I_{k-1}$. Note that the atoms $I\in\cD_{-\infty}$ do not have to cover $\cX$: in fact it is possible that $\cD_{-\infty} = \varnothing$, and sigma algebra $\fS_{-\infty}$ is trivial. 

We denote the collection of all atoms $I\in \cD_{-\infty}$ of finite measure ($|I|<\infty$) by $\cD_{-\infty}\ut{fin}$. 

For $I\in\cD$ denote by $D\ci I$ the martingale difference space $D\ci I :=\Delta\ci I L^2$, and similarly, let $D_k := \Delta_k L^2$. For $I\in \cD_{-\infty}\ut{fin}$ we define 
\[
D\ci I = E\ci I := \E\ci I L^2 =\spn\{\1\ci I\}. 
\]
Define also 
\[
\E_{-\infty} := \sum_{I\in\cD_{-\infty}\ut{fin}} \E\ci I. 
\]

As one can easily see the collection of subspaces $E\ci I$, $I\in\overline\cD$, where 
$\overline\cD:=\cD\cup\cD_{-\infty}\ut{fin}$ is an orthogonal basis in $L^2$. 

We are interesting on the case, when this system is an unconditional basis in the weighted space $L^2(w)$, where $w$ is a weight, i.e.~$w\ge 0$, $w\in L^1(I)$ for all $I\in\cD$. 

Note, that in the case when $\dim D\ci I =\infty$ (i.e. when $I$ has infinitely many children) we have to be a bit more careful, since generally in this case $\Delta\ci I L^2\not \subset L^2(w)$, and they do not need to be closed. We need to introduce the subspaces $\mathring D\ci I:= \Delta\ci I L^\infty$, and then ask when the  the subspaces $D\ci I^w:=\clos\ci{L^2(w)}\mathring D\ci I$ form an unconditional basis in $L^2(w)$. If $\dim D\ci I <\infty$ then $\mathring D\ci I =D\ci I$, so introducing $\mathring D\ci I$ is a moot point in this case. 

Similarly, let us introduce subspaces $\mathring D_k :=\Delta_k (L^1\cap L^\infty)$ and $D_k^w:= \clos\ci{L^2(w)}\mathring D_k$.  . 

The following elementary theorem is an immediate consequence of the general theory of bases. 

\begin{thm}
\label{t:bases-01}
The following statements are equivalent
\begin{enumerate}
\item The system of subspaces $\{D\ci I^w:=\clos\ci{L^2(w)} \mathring D\ci I:\,I\in\overline\cD, \, D\ci I^w\ne\{0\} \}$ is an unconditional basis in $L^2(w)$. 
\item The system of subspaces $\{D_k^w:=\clos\ci{L^2(w)} \mathring D_k:\,-\infty \le k <\infty, D_k^w\ne\{0\} \}$ is an unconditional basis in $L^2(w)$. 
\item The martingale multipliers  $T_\sigma$, $\sigma=(\sigma\ci I )\ci{I\in\cD}$, $\sigma\ci I \in \{0,1\}$ 
\[
T_\sigma f = \sum_{I\in\cD} \sigma\ci I \Delta\ci I f
\]
are uniformly bounded in $L^2(w)$. 
\item The martingale multipliers  $T_\sigma$, $\sigma=(\sigma\ci I )\ci{I\in\cD}$, $|\sigma\ci I |\le1$  
are uniformly bounded in $L^2(w)$. 

\item The martingale multipliers  $T_\alpha$, $\alpha=(\alpha_k )\ci{k\in\Z}$, $\alpha_k \in \{0,1\}$ 
\[
T_\alpha f = \sum_{k\in \Z} \alpha_k \Delta_k f
\]
are uniformly bounded in $L^2(w)$. 

\item The martingale multipliers  $T_\alpha$, $\alpha=(\alpha_k )\ci{k\in\Z}$, $|\alpha_k |\le 1$ 
\[
T_\alpha f = \sum_{k\in \Z} \alpha_k \Delta_k f
\]
are uniformly bounded in $L^2(w)$.
\end{enumerate}
Moreover, the supremums $C\ti{\cond3,\cond4}= \sup_\sigma \|T_\sigma\|\ci{L^2(w)}$ 
from \textup{\cond3} and \textup{\cond4} are equivalent,
\[
C\ti{\cond3} \le C\ti{\cond4} \le 2 C\ti{\cond3}. 
\]
\end{thm}

\begin{rem*}
The supremum $C\ti{cond3} := \sup_\sigma \|T_\sigma\|\ci{L^2(w)}$, where the supremum is taken over all martingale multipliers $T_\sigma$ from \cond3, is what is usually called the \emph{unconditional basis constant} of the system of the martingale difference spaces $D\ci I^w$, $I\in\cD$. 
\end{rem*}

\subsection{Necessity of the \texorpdfstring{$A_2$}{A2} condition}
In what follows we will always assume that $w\not\equiv 0$ on any $I\in \cD_{-\infty}$; otherwise we can just remove the corresponding intervals. 
\begin{prop}
\label{p:nec-01}
Let $w\not\equiv 0$ on any $I\in \cD_{-\infty}$.
The following statements are equivalent
\begin{enumerate}
\item  The system of subspaces $D_k^w:=\clos\ci{L^2(w)} \mathring D_k$, $-\infty \le k <\infty$ is a basis in $L^2(w)$. 
\item The projections $P_{m,n}:= \sum_{k=m}^n \Delta_k$ are uniformly bounded in $L^2(w)$. 
\item The weight $w$ satisfies the following Muckenhoupt $A_2$ condition
\begin{align*}
[w]\ci{A_2} := \sup_{I\in\overline\cD} \La w\Ra\ci I \La w^{-1}\Ra\ci I <\infty.
\end{align*}
where $\overline\cD= \cD\cup\cD_{-\infty}\ut{fin}$. 
\end{enumerate}
Moreover, if the above conditions hold, then
\[
\frac12 [w]\ci{A_2}^{1/2} \le \sup_{m,n\in\Z} \|P_{m,n}\|\ci{L^2(w)} \le 2 [w]\ci{A_2}^{1/2}
\]
\end{prop}
\begin{proof}
The equivalence of \cond1 and \cond2 is just the Banach Basis Theorem. Formally, the Banach Basis Theorem implies that \cond2 is equivalent to the uniform boundedness of $P_{m,n}$ and $\E_{-\infty}$ in $L^2(w)$, but since
\[
\E_{-\infty} f = \lim_{-m,n\to\infty} P_{m,n} f, 
\]
the uniform boundedness of $P_{m,n}$ in $L^2(w)$ implies the estimates for the $\E_{-\infty}$. 

The rest of the proof is based on the well-known fact that 
\[
\|\E\ci I\|\ci{L^2(w)}^2 = \La w\Ra\ci I\La w^{-1} \Ra\ci I, 
\]
see for example \cite{GR}. Therefore $\|\Delta\ci I\|\ci{L^2(w)}\le 2 [w]\ci{A_2}$ and so $\|\Delta_k\|\ci{L^2(w)}\le 2 [w]\ci{A_2}$.

On the other hand, 
if 
\[
A:=\sup_{m,n\in\Z} \|P_{m,n}\|\ci{L^2(w)}, 
\]
then $\|\E_{-\infty}\|\ci{L^2(w)} \le A$. Since 
\[
\E_n f -\E_{-\infty} f = \lim_{m\to-\infty} P_{m,n} f 
\]
we conclude that $\|\E_n\|\ci{L^2(w)}\le 2A$. 
\end{proof}

\section{Main result}
\begin{thm}
\label{t:main-01}
Let $w$ be a weight such that $w\not\equiv 0$ on any $I\in\cD_{-\infty}$. Let $w$
satisfies the martingale $A_2$ condition
\begin{align}
\label{A2-01}
\sup_{I\in\cD} \La w\Ra\ci I  \La w^{-1}\Ra\ci I =:[w]\ci{A_2}<\infty
\end{align}
Then all martingale multipliers $T_\sigma$, $\sigma=(\sigma\ci I)\ci{I\in\cD}$, 
$|\sigma\ci I|\le1$ are uniformly bounded
\[
\|T_\sigma\|\ci{L^2(w)} \le C [w]\ci{A_2}, 
\]
where $C$ is an absolute constant. 
\end{thm}

It is well known that even in the case of Lebesgue reference measure $\nu$ the first power in the $A_2$ constant is sharp, see e. g. \cite{Wi}.

\begin{rem*}
Repeating almost verbatim the extrapolation arguments of \cite{DGPP}, \cite{GR} we can prove
\[
\|T_\sigma\|\ci{L^p(w)} \le C_p [w]\ci{A_p}^{\max (1\frac1{p-1})}\ ,
\]
which is sharp even if the reference measure $\nu$ is  Lebesgue measure. But here we have {\it arbitrary} reference measure $\nu$. It goes without saying that the averages involved in the definition of $[w]\ci{A_p}$ are all taken with respect to $\nu$.
\end{rem*}

\section{Proof of the main result}
\label{s:ProofMainRes}
\subsection{First reductions}
\label{s:FirstReduct}
Fix $\sigma=(\sigma\ci I)\ci{I\in\cD}$. To simplify the notation we skip index sigma and 
use $T$ for $T_\sigma$. 

For $f\in L^2(w)$ we want to estimate $\|Tf\|\ci{L^2(w)}$. If we define $\wt f=:wf$, 
then $f\in L^2(w^{-1})$ and 
\[
\|\wt f\|\ci{L^2(w^{-1})} = \|f\|\ci{L^2(w)}. 
\]
So, if we denote $u=w^{-1}$ and skip\ $\wt{\phantom{a}}$\ over $f$,  the conclusion of  the theorem can be rewritten in the following symmetric form 
\begin{align}
\label{main2w-01}
\| T (fu)\|\ci{L^2(w)} \le C [w]\ci{A_2} \|f\|\ci{L^2(u)} \qquad \forall f\in L^2(u)
\end{align}

The operator $T=T_\sigma$ is well localized, so to prove \eqref{main2w-01} it is sufficient to test $T_\sigma$ on functions $\1\ci I u$ and its adjoint on the functions $\1\ci{I} w$. More precisely, see Theorem \ref{t:T1} below (for homogeneous situation such result for well localized operators was proved in \cite{NTVmrl}), the inequality \eqref{main2w-01} follows from the estimates
\begin{align}
\label{test-01}
\| T\ci I (\1\ci I u)\|\ci{L^2(w)} &\le C [w]\ci{A_2} \|\1\ci I\|\ci{L^2(u)} = C [w]\ci{A_2} 
\left(\La u\Ra\ci I|I| \right)^{1/2}\\
\notag
\|  T\ci I (\1\ci I w)\|\ci{L^2(u)} &\le C [w]\ci{A_2} \|\1\ci I\|\ci{L^2(w)} = C [w]\ci{A_2} 
\left(\La w\Ra\ci I|I| \right)^{1/2}, 
\end{align}
where 
\[
T\ci I =\sum_{I'\in\cD(I)}\sigma\ci{I'} \Delta\ci{ I'}. 
\]


Because of the symmetry, we only need to check the first inequality. 

Define the (very specific) Haar function $h\ci I$ by
\begin{align}
\label{h_I}
h\ci I := \Delta\ci I \1\ci I u = \Delta\ci I u = \sum_{I'\in\ch(I)} (\La u\Ra\ci{I'} - \La u\Ra\ci{I}) \1\ci{I'}\,.
\end{align}

Using $h\ci I$s we can write 
\[
T\ci{I_0}(\1\ci{I_0}u) = \sum_{I\in\cD:\,I\subset I_0} \sigma\ci I  h\ci I
\]

Define the corresponding weighted Haar functions $h\ci I^w$
\[
h\ci I^w = h\ci I - \gamma\ci I^w \1\ci I, 
\]
where $\gamma\ci I^w$ is chosen to make $h\ci I^w$ orthogonal to constants in $L^2(w)$ (equivalently, to satisfy  $\La h\ci I^ww\Ra\ci I=0$). Direct calculations give us 
\begin{align}
\notag
\gamma\ci I^w & = \La w\Ra\ci I^{-1} |I|^{-1}\sum_{I'\in\ch(I)} \left(\La u\Ra\ci{I'} - \La u\Ra\ci{I}\right) \La w\Ra\ci{I'} |I'| 
\\
\label{gamma_I^w}
& =\La w\Ra\ci I^{-1} \sum_{I'\in\ch(I)} \left(\La u\Ra\ci{I'} - \La u\Ra\ci{I}\right) \left(\La w\Ra\ci{I'} - \La w\Ra\ci{I}\right) |I'|/|I| 
\end{align}
Note that functions $h\ci I^w$ form an orthogonal system in $L^2(w)$, and that 
\[
\|h\ci I\|\ci{L^2(w)}^2 = \|1\ci I\|\ci{L^2(w)}^2 + \|h\ci I^w\|\ci{L^2(w)}^2, 
\] 
so $\|h\ci I^w\|\ci{L^2(w)}^2 \le \|h\ci I\|\ci{L^2(w)}^2$. Therefore 
\begin{align*}
\Biggl\| \sum_{I\in \cD:\, I \subset I_0}  \sigma\ci I  h\ci I^w \Biggr\|^2_{L^2(w)} 
&=  \sum_{I\in \cD:\, I \subset I_0}  
|\sigma\ci I|^2  \|h\ci I^w\|\ci{L^2(w)}^2 \\
&
\le \sum_{I\in \cD:\, I \subset I_0}  
 \|h\ci I\|\ci{L^2(w)}^2 .
\end{align*}
So, if we show that 
\begin{align}
\label{test-03}
\sum_{I\in \cD:\, I \subset I_0} \| h\ci I \|\ci{L^2(w)}^2 \le C [w]\ci{A_2}^2 \|\1\ci{I_0}\|\ci{L^2(u)}^2
\intertext{and that}
\label{test-04}
\Biggl\| \sum_{I\in \cD:\, I \subset I_0}  \sigma\ci I  \gamma\ci I^w \1\ci I \Biggr\|_{L^2(w)} \le C [w]\ci{A_2} \|\1\ci{I_0}\|\ci{L^2(u)}
\end{align}
we get, using triangle inequality that 
\[
\Biggl\| \sum_{I\in \cD:\,I\subset I_0} \sigma\ci I h\ci I \Biggr\|_{L^2(w)} \le C [w]\ci{A_2} \|\1\ci{I_0} \|\ci{L^2(w)}, 
\] 
i.e.~that \eqref{test-01} holds for $I=I_0$. 

Using \eqref{h_I} we can rewrite \eqref{test-03} as 
\begin{align}
\label{test-05}
\sum_{I\in \cD:\, I \subset I_0} \sum_{I'\in\ch(I)} | \La u\Ra\ci{I'} - \La u\Ra\ci{I} |^2 \La w\Ra\ci{I'} |I'|\le C [w]\ci{A_2}^2 \La u\Ra\ci{I_0} |I_0|. 
\end{align}

By duality, \eqref{test-04} follows from the estimate
\[
\sum_{I\in \cD:\, I \subset I_0} |\gamma\ci I^w| \cdot|\La gw\Ra\ci I |\cdot |I| 
\le C [w]\ci{A_2} \|\1\ci{I_0}\|\ci{L^2(u)} \|g\|\ci{L^2(w)} \qquad \forall f \in L^2(w). 
\]
Recalling formula \eqref{gamma_I^w} for $\gamma\ci I^w$ we can rewrite this estimate as 
\begin{align}
\label{test-06}
\sum_{I\in \cD:\, I \subset I_0} \frac{|\La gw\Ra\ci{I}|}{\La w\Ra\ci{I}}\rho\ci I  \cdot |I| \le C [w]\ci{A_2} \La u\Ra\ci{I_0}^{1/2} |\La g^2w\Ra\ci{I_0}|^{1/2}|I_0|. 
\end{align}
where
\[
\rho\ci I:= \sum_{I'\in\ch(I)}  | \La u\Ra\ci{I'} - \La u\Ra\ci{I} |\cdot 
| \La w\Ra\ci{I'} - \La w\Ra\ci{I} |\cdot|I'|/|I|
\]

So, we reduced our main theorem to the estimates \eqref{test-05} and \eqref{test-06}. 

\subsection{Outer measure spaces and reduction to Carleson measure properties}
\label{OMS}
We will reduce the estimates \eqref{test-05}, \eqref{test-06} to checking the Carleson measure properties for some sequences. 
Outer measure spaces give us a very convenient language for doing that. 
We present here some basic facts about such spaces: a reader interested in more details should consult the paper \cite{DT}. 

\subsubsection{Outer measure spaces: basic facts} For $I\in\cD$ denote $\cD(I):=\{I'\in\cD: I'\subset I\}$. 

For a measure $\mu$ on $\cX$ define its \emph{outer lifting} $\mu^*$ to $\cD$ to be the outer measure defined of the sets $\cD(I)$ by 
\[
\mu^*(\cD(I)) =\mu(I). 
\]
The outer measure $\mu^*$ extends to arbitrary $\cA\subset \cD$ by the usual recipe: one considers all collections $\cK\subset \cD$ such that $\cA\subset \bigcup_{I\in\cK} \cD(I)$ and put 
\[
\mu^*(\cA) =\inf_{\cK} \sum_{I\in\cK}\mu^*(\cD(I))
\]
where the infimum is taken over all such collections $\cK$.

To define the outer measure spaces used below we need to introduce the so-called \emph{size} function (a generalization of the local square function). Namely, for a measure $\mu$ on $\cX$, $p\in[1,\infty)$, and a function $F$ on $\cD$  define the size function $S_\mu^p F$ on the collection of sets $\cD(I)$, $I\in\cD$  by
\begin{align*}
S^p_\mu F(\cD(I)) := \left(\mu(I)^{-1} \sum_{I'\in\cD(I)} |F(I')|^p \mu(I')   \right)^{1/p}\,.
\end{align*}
If $\mu$ is the underlying measure $\nu$ we skip the subscript and use $S^p F$; if $d\mu=wd\nu$ we use the notation $S_w^pF$. 

We will also need the size function $S^\infty_\mu$, 
\[
S^\infty_\mu F(\cD(I)) = \sup_{I'\in\cD(I):\, \mu(I')>0} 
|F(I')| . 
\]

For a size function $S$, where $S$ is one of the size functions $S^p_\mu$, $S^\infty$ defined above, 
define the outer $L^\infty(S)$ norm on functions on $\cD$ by 
\[
\|F\|\ci{L^\infty(\cD,S)} := \sup_{I\in\cD} SF(\cD(I));
\]
the space $L^\infty(\cD,S)$ consists of all functions on $\cD$ for which this norm is finite.

To define the $L^p$ space $L^p(\cD, \mu^*, S)$  we need to define what is the outer measure $\mu^*$ of the \emph{superlevel set} $\{SF>\lambda\}=\{\cD(I): SF(\cD(I))>\lambda\}$. To do that we 
consider  all $\cG\subset \cD$ such that $S(F\1\ci{\cD\setminus\cG}) \le \lambda$ (on all $\cD(I)$) and take the infimum of $\mu^*$ of such $\cG$. Formally, 
\[
\mu^*( SF>\lambda ):= \inf\{   \mu^*(\cG): \cG\subset \cD, \, S(F\1\ci{\cD\setminus\cG}) \le \lambda \}.
\]
Then for  a function $F$ on $\cD$ we define
\[
\|F\|\ci{L^p(\cD, \mu^*,\,S)} = \left( p\int_0^\infty \lambda^{p-1} \mu^*(SF>\lambda) d\lambda \right)^{1/p}
\]
If the outer measure $\mu^*$ is the outer lifting of the ambient measure $\nu$ we will skip the measure and use the notation $L^2(\cD, S)$.

We will need the following simple fact. It is a particular instance of the Radon--Nikodym property for outer measure spaces from \cite{DT}.
\begin{lm}[$L^1$--$L^\infty$ duality]
\label{l:L1-L^infty}
Let $F \in L^1(\cD,\mu^*,\, S^\infty_\mu)$ and $G\in L^\infty(\cD,  S^1_\mu)$. Then 
\[
\sum_{I\in\cD} | F(I)G(I)| \mu(I)\le \|F\|\ci{L^1(\cD, \mu^*,\,S^\infty)} \|G\|\ci{L^\infty(\cD,  S_\mu^1)}. 
\]
\end{lm}
\begin{proof}
By homogeneity we can assume that $\|G\|\ci{L^\infty(\cD, S_\mu^1)}=1$. Let us treat $|G|$ as a measure on $\cD$. Namely, for $\cK\subset \cD$ define its measure  $\cM(\cK)$  as
\[
\cM(\cK):= \sum_{I\in\cK} |G(I)| \mu(I). 
\]

For $\lambda>0$ let $\cJ_\lambda$ be the collection of maximal intervals $I\in\cD$ for which $|F(I)|>\lambda$ (we can always assume that $F$ has finite support, so  maximal intervals do exist).

Note that 
\begin{align*}
\{|F|>\lambda\} &\subset \bigcup_{I\in\cJ_\lambda} \cD(I), 
\intertext{and that}
\mu^*(S^\infty F>\lambda) & = \sum_{I\in\cJ_\lambda} \mu^*(\cD(I))\ .
\end{align*}
The last equality holds because for any $J\in \cD\setminus \bigcup_{I\in\cJ_\lambda} \cD(I)$ we have $|F(J)|\le \lambda$, and we cannot throw away any $I$ from $\cJ_\lambda$ and still have the same property.

Therefore, 
\begin{align*}
\cM(\{|F|>\lambda\}) & \le \sum_{I\in\cJ_\lambda} \sum_{I'\in\cD(I)}|G(I')|  \mu(I')
\\
&\le \sum_{I\in\cJ_\lambda}  \|G\|\ci{L^\infty(\cD, S_1)} \mu(I) 
\le \sum_{I\in\cJ_\lambda} \mu(I)
\\
& =\sum_{I\in\cJ_\lambda} \mu^*(\cD(I)) =  \mu^*(S^\infty F >\lambda)\ .
\end{align*}

Integrating both side with respect to $d\lambda$ we get the conclusion of the lemma.
\end{proof}

\begin{lm}
\label{l:1over-01}
Let $h$ be a positive function on $I_0\in\cD$, $h, h^{-1}\in L^1(I_0)$. Consider the function $H$ 
on $\cD$ given by the formula 
\[
H(I) =\left\{
\begin{array}{ll} 1/\La h^{-1}\Ra\ci I, \qquad &  I\in \cD(I_0) ;\\
0, & I\notin \cD(I_0) .\end{array}\right. 
\] 
Then $H\in L^1(\cD,  S^\infty)$, and 
\[
\|H\|\ci{L^1(\cD, S^\infty)}\le 2 \|h\|_\nu. 
\]
\end{lm}
\begin{proof}

Fix $\lambda>0$ and denote by $\cH_\lambda$ the collection of 
maximal  intervals $I\in\cD$ such that $H(I)> \lambda$.
For such maximal $I$ we define
\[
E_I:=\left\{x\in I: h^{-1}(x) \le 2\La h^{-1}\Ra\ci I  \right\}\ .
\]
Then $|E\ci I| \ge \frac12 |I|$. On the other hand, on our maximal $I$, we have
\[
x\in E\ci I \Rightarrow \lambda < \frac1{\La h^{-1}\Ra\ci I} \le 2 h(x)\ .
\]
This inequality implies
\begin{align*}
\nu^*(S^\infty(H)>\lambda) \le \Bigl| \bigcup_{I\in \cH_\lambda} I\Bigr| \le 2\sum_{I\in \cH_\lambda} |E\ci I| =2\Bigl|\bigcup_{I\in \cH_\lambda} E\ci I \Bigr|  \le 2 \left|\{x: h(x)> \lambda/2\}\right|\ .
\end{align*}
Integrating  both sides of the inequality with respect to $d\lambda$ proves the lemma.
\end{proof}

\subsection{Averaging operators and outer measure spaces.}
\label{averagin&OMS}
For a measure $\mu$ on $\cX$ define the averaging operator $\A_\mu$ transforming functions on $\cX$ to functions on $\cD$, 
\[
\A_\mu f(I) := \mu(I)^{-1} \int_I f\,d\mu =: \La f \Ra\ci{I,\mu} , \qquad I\in \cD. 
\]
Define also the maximal operator $M_\mu =M_\mu\ut d$, 
\begin{align}
\label{MaxOp}
M_\mu\ut d f (x) :=\sup_{I\in\cD:\, x\in I} \mu(I)^{-1} \int_I |f|\,d\mu
\end{align}
If $d\mu = wd\nu$, where $\nu$ is the underlying measure, we will use notation $\A_w$, $M_w$. 
\begin{lm}
\label{l:bilin}
Let $w$ be a weight such that $w, w^{-1}\in L^1(I)$ for all $I\in\cD$. 
Then the bilinear operator $f\times g\mapsto \A_{w^{-1}} f \A_{w}g$ is bounded from $L^2(w^{-1})\times L^2(w)$ to $L^1(\cD, S^\infty)$
\[
\| \A_{w^{-1}} f \A_{w} g \|\ci{L^1(\cD, S^\infty)} \le 
4\|f\|\ci{L^2(w^{-1})}\|g\|\ci{L^2(w)}.
\]
\end{lm}
\begin{proof}
It suffices to prove the lemma for positive functions $f$ and $g$. We can also assume that $f$ and $g$ are supported on a union of finitely many intervals. 
Let us consider the collection $\cJ_\lambda$ of all maximal  intervals $I\in\cD$ such that
\[
\A_{w^{-1}} f(I) \A_{w}g(I) > \lambda  .
\]
Then for any such maximal interval $I$ and any $x,y\in I$:
\begin{align}
\label{M>FG}
(M\ut d_{w^{-1}} f)(x) (M\ut d_{w}g)(y) \ge \A_{w^{-1}} f(I) A_{w}g(I) >\lambda\ .
\end{align}
Clearly \eqref{M>FG} implies that 
\[
\{ M^d_{w^{-1}} f\cdot M^d_{w} g>\lambda\} \supset \bigcup_{I\in\cJ_\lambda} I\,.
\]

Denote  $\Phi:=  M^d_{w^{-1}} f\cdot M^d_{w} g$.
We use now that the set $\{S^\infty(\A_{w^{-1}} f \, \A_{w} g)>\lambda \}$ is  exactly the union $\bigcup_{I\in\cJ_\lambda} \cD(I)$, and we write
\[
\nu^*(S^\infty(\A_{w^{-1}} f \A_{w} g)>\lambda )\le \sum_{I\in \cJ_\lambda} |I|
\le \left| \{x\in\R: \Phi(x)>\lambda\} \right|  \, .
\]
Integrating with respect to $\lambda$ we then get that
\[
\| \A_{w^{-1}} f \A_{w} g \|\ci{L^1(\cD, S^\infty)} \le \int_\cX M\ut d_{w^{-1}} f\cdot M\ut d_{w} g \,d\nu\ .
\]
But we know that the martingale maximal operator is bounded in $L^2$, namely for any $\mu$
\[
\|M_\mu\ut d f \|\ci{L^2(\mu)} \le 2 \|f\|\ci{L^2(\mu)}. 
\] 
Therefore, 
\begin{align*}
 \int  (M\ut d_{w^{-1}} f\cdot M\ut d_{w} g) d\nu & = \int (M\ut d_{w^{-1}} f) w^{-1/2} (M\ut d_{w} g) w^{1/2} d\nu 
\\
&\le \bigg(\int (M^d_{w^{-1}} f)^2 w^{-1} d\nu \bigg)^{1/2}\bigg(\int (M^d_{w} g)^2 w d\nu\bigg)^{1/2}
\\
& \le 4\|f\|\ci{L^2(w^{-1})}\|g\|\ci{L^2(w)}.
\end{align*} 
This completes the proof of Lemma \ref{l:bilin}.
\end{proof}

\subsubsection{Reduction to the Carleson measure properties}
\label{reduCarl}

We now reduce Theorem \ref{t:main-01} to the following two lemmas, which will be proved in the next section. 

\begin{lm}
\label{l:CMP-01}
The collection $\tau=(\tau\ci I)\ci{I\in\cD}$, 
\[
\tau\ci I = \sum_{I'\in\ch(I)} | \La u\Ra\ci{I'} - \La u\Ra\ci{I} |^2 \La w\Ra\ci{I}\La w\Ra\ci{I'} |I'|/|I|
\] 
satisfies the Carleson measure property
\begin{align}
\label{CMP-01}
\sum_{I\in\cD:\,I\subset I_0} \tau\ci I |I| \le C [w]\ci{A_2}^2 |I_0| \qquad \forall I_0\in\cD
\end{align}
with an absolute constant $C$. 
\end{lm}

\begin{lm}
\label{l:CMP-02}
The collection $\rho=(\rho\ci I)\ci{I\in\cD}$,
\[
\rho\ci I:= \sum_{I'\in\ch(I)}  | \La u\Ra\ci{I'} - \La u\Ra\ci{I} |\cdot 
| \La w\Ra\ci{I'} - \La w\Ra\ci{I} |\cdot|I'|/|I|
\]
satisfies the following Carleson measure property
\begin{align}
\label{CMP-02}
\sum_{I\in\cD:\,I\subset I_0} \rho\ci I |I| \le C [w]\ci{A_2} |I_0| \qquad \forall I_0\in\cD
\end{align}
with an absolute constant $C$. 
\end{lm}

Let us show that these lemmas imply Theorem \ref{t:main-01}. We already reduced the theorem to proving the estimates \eqref{test-05} and \eqref{test-06}. We can write right hand side of \eqref{test-05} as
\begin{align*}
\sum_{I\in \cD:\, I \subset I_0} \La w \Ra\ci I^{-1}\sum_{I'\in\ch(I)} | \La u\Ra\ci{I'} - \La u\Ra\ci{I} |^2 \La w \Ra\ci I \La w\Ra\ci{I'} |I'| = \sum_{I\in \cD:\, I \subset I_0} \La w \Ra\ci I^{-1} \tau\ci I |I| \,.
\end{align*}
Lemma \ref{l:1over-01} applied to $h=u\1\ci{I_0}$ implies that the function $F$ on $\cD$
\[
F(I)  =\left\{
\begin{array}{ll} 1/\La w \Ra\ci I, \qquad &  I\in \cD(I_0) ;\\
0, & I\notin \cD(I_0) ,\end{array}\right.  
\]
belongs to $L^1(\cD, S^\infty)$ and that 
\begin{align}
\label{normF}
\| F\|\ci{L^1(\cD, S^\infty)} \le 2 \|u\1\ci{I_0}\|\ci{L^1} = 2 \La u \Ra\ci{I_0}\,.
\end{align}

On the other hand the Carleson measure property \eqref{CMP-01} of Lemma \ref{l:CMP-01} means that the function $G$ on $\cD$, 
\[
G(I)= \tau\ci I \,, \qquad I\in \cD
\]
belongs to $L^\infty(\cD, S^1)$, and that 
\begin{align}
\label{normG}
\| G\|\ci{L^\infty(\cD, S^1)}  \le C[w]\ci{A_2}^2 \,.
\end{align}
Combining Lemma \ref{l:L1-L^infty} with estimates \eqref{normF} and \eqref{normG} we immediately get \eqref{test-05}.

To prove \eqref{test-06} we first apply Lemma \ref{l:bilin} with $f=\1\ci{I_0}$ and $\wt g:=g\1\ci{I_0}$ instead of $g$. Then we get that the function $F$ on $\cD$, 
\[
F(I) :=
 \A_u f (I) \A_w \wt g(I) = 
{\La \1\ci{I_0}g w\Ra\ci I}/{\La  w\Ra\ci I} , \qquad   I\in\cD
\]
(recall that $u:=w^{-1}$) belongs to $L^1(\cD, S^\infty) $ and that 
\[
\|F \|\ci{ L^1(\cD, S^\infty) } \le 4\|f\|\ci{L^2(u)} \|\wt g\|\ci{L^2(w)} = 4 \La u \Ra\ci{I_0}^{1/2} \La gw\Ra\ci{I_0}^{1/2} |I_0| \,.
\]

Lemma \ref{l:CMP-02} implies that the function $G$ on $\cD$, $G(I) := \rho\ci I$, $I\in\cD$, is in $L^\infty(\cD, S^1)$ and that 
\[
\|G\|\ci{L^\infty(\cD, S^1)} \le C [w]\ci{A_2}^2 \,.
\]
Combining Lemma \ref{l:L1-L^infty} with the above estimates and summing only over all $I\in\cD(I_0)$ we get the desired estimate \eqref{test-06}. \hfill \qed

So, we reduced proof of Theorem \ref{t:main-01} to proving Lemmas \ref{l:CMP-01} and \ref{l:CMP-02}.

\section{Bellman functions and the proof of the Carleson measure properties}

For a smooth function $\cB$ (defined on an open convex set $\Omega\subset\R^d$) and a point $X_0\in \Omega$, define $\cB_{X_0}$ as 
\begin{align}
\label{B_x0} 
\cB_{X_0}(x) = \cB(X) - \cB'(X_0) (X-X_0). 
\end{align}

We need the following	 trivial Lemma. 

\begin{lm}
\label{l:conc-triv}
Let $\cB$ be a smooth function defined in an open convex set $\Omega$, and let $X_k\in \Omega$, $k\ge 0$ satisfy
\[
X_0=\sum_{k\ge1} \theta_k X_k, \qquad \theta_k \ge 0, \quad \sum_k\theta_k =1. 
\]
Then 
\begin{align}
\label{conc-comb}
\cB(X_0) - \sum_{k\ge1} \theta_k \cB(X_k) = \sum_{k\ge 1} \theta_k (\cB_{X_0}(X_0) - \cB_{X_0} (X_k)) . 
\end{align}
\end{lm}
\begin{proof}
This is obvious.
\end{proof}

In what follows we will use $Q=[w]\ci{A_2}$, but for now $Q$ is just an arbitrary constant, $Q\ge1$.

Define 
\[
\Omega\ci Q := \{ X=(x,y)\in \R^2: 0< xy \le Q\}.
\]
\begin{lm}
\label{l:bell-01}
Let 
\[
\cB(X)=\cB(x,y) := 4Q^{1/2} (xy)^{1/2} - xy.
\]
Then
\begin{enumerate}
\item $0 \le \cB(X) \le 4 Q$ for all $X\in \Omega\ci Q$.

\item For $X_0, X \in \Omega\ci Q$
\begin{align}
\label{MDI-01}
 \cB\ci{X_0} (X_0) - \cB\ci{X_0} (X) \ge c |x-x_0|\cdot |y-y_0|, 
\end{align}
where $c$ is an absolute constant. 
\end{enumerate}
\end{lm}

\begin{lm}
\label{l:bell-02}
Let 
\[
\cB(X)=\cB(x,y) := 128 Q^{3/2} (xy)^{1/2} - (xy)^2.
\]
Then
\begin{enumerate}
\item $0 \le \cB(X) \le 128 Q^2$ for all $X\in \Omega\ci Q$.

\item For $X_0, X \in \Omega\ci Q$
\begin{align*}
 \cB\ci{X_0} (X_0) - \cB\ci{X_0} (X) \ge c |x-x_0|^2 yy_0, 
\end{align*}
where $c$ is an absolute constant. 
\end{enumerate}
\end{lm}

Using these lemmas we can prove Lemmas \ref{l:CMP-02} and \ref{l:CMP-01} by applying the standard Bellman function technique.  In fact,  let us plug into inequalities of Lemmas \ref{l:bell-01} and \ref{l:conc-triv} the following data: $x_0= \La u\Ra_I, y_0=\La w\Ra_I,
x_k =\La u\Ra_{I'}, y_k =\La w\Ra_{I'}$, where $I'$ is the $k$-th children of $I$ (enumeration is not important). Now look at the conjunction of   \eqref{conc-comb}, \eqref{MDI-01} in this new form. Clearly $\theta_k=\frac{|I'|}{|I|}$. Multiply the resulting inequality by $|I|$.

Then  we will get
\[
\rho\ci I |I|= \sum_{I'\in\ch(I)}  | \La u\Ra\ci{I'} - \La u\Ra\ci{I} |\cdot 
| \La w\Ra\ci{I'} - \La w\Ra\ci{I} |\cdot|I'| \le |I| B(X_0) -\sum_{I'\in \ch(I)} |I'| B(X_k)\ .
\]
Here $X_0=(x_0, y_0) = (\La u\Ra_I, \La w\Ra_I)$ and $X_k=(x_k, y_k)= (\La u\Ra_{I'},\La w\Ra_{I'})$, where $I'$ is the $k$-th children of $I$ (enumeration is not important).

Notice that in the right-hand-side we have a telescopic term. If we start to sum up this inequality  over $I\subset I_0$ we get immediately Lemma \ref{l:CMP-02}, if we use the function $\cB$ from Lemma \ref{l:bell-01}. Lemma \ref{l:CMP-02} is then proved.

To prove Lemma \ref{l:CMP-01} we repeat this argument varbatim, but now we use the conjunction of  Lemmas \ref{l:bell-02} and \ref{l:conc-triv} and we use  function $\cB$ from Lemmas \ref{l:bell-02}.

The Lemmas \ref{l:bell-01} and \ref{l:bell-02} will be proved in the next section. Along with the outer measure spaces these lemmas are the main tools of our proof.
 
\section{Investigation of Bellman functions}
\label{BellmanF}

In this section we will prove Lemmas \ref{l:bell-01} and \ref{l:bell-02}. 

\subsection{Proof of Lemma \ref{l:bell-02}}
Statement \cond1 of the lemma is obvious, only statement \cond2 needs proving. 

\subsubsection{Preliminaries}
We will prove a stronger inequality, namely
\begin{align}
\label{MDI-mod-01}
\cB\ci{X_0} (X_0)- \cB\ci{X_0}(X) \ge c(yy_0 (\sd x)^2 + xx_0 (\sd y)^2), 
\end{align}
where $\sd x:= x-x_0$, $\sd y: = y-y_0$.

Define $\cB{[1]}(X) = (xy)^{1/2}$, $\cB^{[2]}(X) := -(xy)^2$, so 
\[
\cB=128 Q^{3/2} \cB^{[1]}+\cB^{[2]} . 
\]
Note that for $x, y>0$
\begin{align}
\label{d2B_1}
-d^2\cB^{[1]}:= \frac14 \left(\begin{array}{c} {dx}\\ {dy}\end{array}\right)^T \left(\begin{array}{cc}  x^{-3/2}y^{1/2} & -x^{-1/2}y^{-1/2} \\
-x^{-1/2}y^{-1/2} &  x^{1/2} y^{-3/2} \end{array} \right)
\left(\begin{array}{c} {dx}\\ {dy}\end{array}\right) \ge 0
\end{align}
so the function $\cB^{[1]}$ is concave for all $x,  y>0$. 

Computing $d^2\cB^{[2]}$ we get 
\begin{align}
\label{d2B_2}
-d^2\cB^{[2]} = 2 \left(\begin{array}{c} {dx}\\ {dy}\end{array}\right)^T
\left(\begin{array}{cc}  y^2 & 2xy \\
2xy &   x^2 \end{array} \right)
\left(\begin{array}{c} {dx}\\ {dy}\end{array}\right) \,.
\end{align}
Note, that the function $\cB^{[2]}$ is not concave. 

Denote 
\[
X(t):=(x(t),y(t)) = (x_0 + t\sd x, y_0+t\sd y)= X_0 + t\sd X, \qquad 0\le t\le 1, 
\] 
and let $\Phi\ci{\ssd X} (t) := \cB (X_0+t\sd X)$, $\Phi^{[1,2]}\ci{\ssd X} (t) =\cB^{[1,2]}(X(t))$. 

Recall that by the integral form of Taylor's remainder we have for a function $\f$ on an interval
\begin{align}
\label{TRF-I}
\f(x) - \f(x_0) - \f'(x_0) (x-x_0) = \int_{x_0}^x \f''(t) (x-t) dt . 
\end{align}

\subsubsection{The easy cases} 
If $\sd x \sd y \ge 0$, we conclude using \eqref{d2B_2} and \eqref{d2B_1} (together with the fact that $-d^2\cB_1\ge 0$) that 
\begin{align}
\label{Phi''easy}
-\Phi\ci{\ssd X}''(t) \ge y(t)^2(\sd x)^2 + x(t)^2 (\sd y)^2, 
\end{align}
Recalling that $x(t) =x_0 +t\sd x$, $y(t)= x_0+t\sd y$ we get
\begin{align}
\label{int-x^2}
\int_0^1 x(t)^2 (1-t) dt & \ge c (x_0+\sd x/2)^2 \ge c x_0 x,  \\
\label{int-y^2}
\int_0^1 y(t)^2 (1-t) dt & \ge c (y_0+\sd y/2)^2 \ge c y_0 y;
\end{align}
the constant $c$ can be computed explicitly. The above inequalities prove \eqref{MDI-mod-01} (for the case $\sd x \sd y\ge0$).

Let us now treat the case $\sd x\sd y<0$. Because of the symmetry of \eqref{MDI-mod-01} we can assume without loss of generality that $\sd x>0$ and $\sd y<0$. 

Consider first the simple case when 
\begin{align}
\label{xy<4Q}
x(t)y(t) \le 4 Q\qquad \text{ for all } t\in[0,1]. 
\end{align}
We will discuss later for which $\sd x$ and $\sd y$ this happens, but for now we will continue with the estimates. Under the assumption  \eqref{xy<4Q} we have 
\[
8 Q^{3/2} (x(t)y(t))^{-1/2} \ge x(t)y(t) 
\]
so the term including $\sd x \sd y$ in $\Phi\ci{\ssd X}''(t)$ is non-negative. But that means the estimate \eqref{Phi''easy} holds in this case as well, so we again get the conclusion using \eqref{int-x^2} and \eqref{int-y^2}.  

We claim that  \eqref{xy<4Q} holds if either $\sd x/x_0\le 3$ or $-\sd y/y_0\le 1/2$ (all under the assumption $\sd x>0$, $\sd y<0$, $xy \le Q$, $x_0y_0\le Q$). Indeed, if $\sd x/x_0\le 3$, then $x(t)\in [x_0, 4x_0]$. Since $\sd y<0$, we get that $y(t)\le y_0$, so 
\[
x(t)y(t) \le 4x_0 y_0 \le 4Q. 
\] 
On the other hand, if $-\sd y /y_0\le 1/2$, then $y(t)\le 2 y$, and since $x(t) \le x$ we get
\[
x(t)y(t)\le 2x(t) y \le 2xy \le 2Q\le 4Q.
\]

\subsubsection{The hard case}
So, it remains to investigate the hard case 
\begin{align}
\label{hard-cond}
\frac{\sd x}{x_0}\ge 3, \qquad  -\frac{\sd y}{y_0} \ge \frac12.
\end{align}
Denote $a:=\sd x/x_0$, $b=-\sd y/y_0$ and consider the function $\f$, $\f(t) = \bigl( (1+at)(1-bt)\bigr)^2$. We can write
\begin{align*}
\f(t)^2 & = \bigl(1+t(b-a) +t^2 ab\bigr)^2 
\\
&
= 1+t^2(a-b)^2 + t^4(ab)^2 + 2t (a-b) -2t^2 ab - 2t^3(a-b)ab. 
\end{align*}
Subtracting linear term $2t (a-b)$ we get for $t=1$
\begin{align*}
\f(1) - \f(0) - \f'(0)\cdot 1 & = a^2 + b^2 -2ab + (ab)^2 -2ab - 2a^2 b + 2 ab^2 \\
&= a^2 (1-b)^2 + b^2 -4ab + 2ab^2 \ge -4ab. 
\end{align*}
Multiplying this inequality by $(x_0y_0)^2$ and recalling the definition of $a$ and $b$ we get 
\begin{align}
\notag
\Phi^{[2]}\ci{\ssd X}(0) + (\Phi^{[2]}\ci{\ssd X})'(0)\cdot 1 - \Phi^{[2]}\ci{\ssd X}(1) & \ge 4 x_0y_0 \sd x \sd y, 
\intertext{or, equivalently}
\label{Loss-B2}
\cB^{[2]}\ci{X_0} (X_0) - \cB^{[2]}\ci{X_0}(X) & \ge 4 x_0y_0 \sd x \sd y. 
\end{align}
Note that the term in the right hand side has the wrong sign: it is negative, and we would like to have an estimate  below by a positive quantity. But we will show, that all the ``damage'' done by this term will be compensated by what we gain from $128 Q^{3/2}\cB^{[1]}$.

Using the Taylor remainder formula \eqref{TRF-I} we get 
\begin{align*}
\cB^{[1]}\ci{X_0} (X_0) - \cB^{[1]}\ci{X_0} (X) = \int_{0}^1 (-\Phi\ci{\ssd X}^{[1]})''(t) \cdot(1-t) \, dt .
\end{align*}
Since $\sd x\sd y<0$ the off-diagonal terms in the Hessian \eqref{d2B_1} give us a non-negative contribution, so 
\[
(-\Phi\ci{\ssd X}^{[1]})''(t) \ge (\sd x)^2 y(t)^{1/2} x(t)^{-3/2}. 
\]

We can estimate
\begin{align*}
\frac1{(\sd x)^2}\int_{0}^1 (-\Phi\ci{\ssd X}^{[1]})''(t) \cdot(1-t) \, dt & \ge \frac12
\int_{0}^{1/2} (-\Phi\ci{\ssd X}^{[1]})''(t)  \, dt 
\\
&\ge \frac12\int_{0}^{1/2} y(t)^{1/2} x(t)^{-3/2}\, dt
\\
&\ge  \frac{y(1/2)^{1/2}}{2}\int_{0}^{1/2} (x_0 + t\sd x)^{-3/2}  \, dt 
\\
& = - {y(1/2)^{1/2}} (x_0 + t\sd x)^{-1/2} (\sd x)^{-1} \Bigm|_{t=0}^{t=1/2} 
\\
&  = \alpha {y(1/2)^{1/2}} x_0^{-1/2} (\sd x)^{-1};
\end{align*}
here  $\alpha = 1- (1+3/2)^{-1/2} = 1-(2/5)^{1/2}$. 
We can estimate  
\begin{align}
\notag
\cB^{[1]}\ci{X_0} (X_0) - \cB^{[1]}\ci{X_0} (X) & \ge 2^{-1/2} \alpha y_0^{1/2} x_0^{-1/2} (\sd x)^{-1} (\sd x)^2 
\\
\label{gainB_1}
& \ge \frac14  y_0^{1/2} x_0^{-1/2} \sd x
\end{align}
Therefore, since $x_0y_0\le Q$
\begin{align*}
16 Q^{3/2} (\cB^{[1]}\ci{X_0} (X_0) - \cB^{[1]}\ci{X_0} (X)) & \ge 4 y_0^{1/2} x_0^{-1/2} \sd x 
\\
&  = 4 Q^{3/2} y_0^{1/2}|\sd y|^{-1} x_0^{-1/2} |\sd x\sd y|  \\
& \ge
4 Q^{3/2} y_0^{-1/2} x_0^{-1/2} |\sd x\sd y|\ge 
\\
& \ge 4 (x_0 y_0)^{3/2} y_0^{-1/2} x_0^{-1/2} |\sd x\sd y| = 4 x_0y_0 |\sd x\sd y|\,,
\end{align*}
so negative contribution \eqref{Loss-B2} of $\cB^{[2]}$ to \eqref{MDI-mod-01} is compensated by the contribution of the term $16Q^{3/2}\cB^{[1]}$. We then have the contribution of the term $112 Q^{3/2}\cB^{[1]}$ remaining. the contribution of  of the term $48 Q^{3/2}\cB^{[1]}$ gives us   by \eqref{gainB_1}  
\begin{align}
\notag
48 \left(\cB^{[1]}\ci{X_0} (X_0)- \cB^2{[1]}\ci{X_0}(X) \right) & \ge 12  Q^{3/2} y_0^{1/2} x_0^{-1/2} \sd x
\\
\notag
&\ge 12 xy \, (x_0y_0)^{1/2} y_0^{1/2} x_0^{-1/2} \sd x 
\\
\label{MDI-mod-02}
& = 12 yy_0 x \sd x\ge 12 y y_0 (\sd x)^2;
\end{align}
in this estimate we used both conditions $xy\le Q$ and $x_0y_0\le Q$. 

Let us now estimate the contribution of the remaining $64 Q^{3/2} \cB^{[1]}$ a bit differently. Again using \eqref{gainB_1} we get 
\begin{align}
\notag
64 \left(\cB^{[1]}\ci{X_0} (X_0)- \cB^2{[1]}\ci{X_0}(X) \right)  
& \ge 16  Q^{3/2} y_0^{1/2} x_0^{-1/2} \sd x
\\
\notag 
& \ge 16  Q^{3/2} y_0^{1/2} x_0^{-1/2} (\sd x)\, y_0^{-2} (\sd y)^2
\\
\notag 
& \ge 16 Q^{3/2} y_0^{-3/2} x_0^{-1/2} (3 x/4)\,(\sd y)^2 
\\
\notag 
& \ge 12 (x_0y_0)^{3/2}  y_0^{-3/2} x_0^{-1/2} x \,(\sd y)^2  \\ 
\label{MDI-mod-03}
& \ge 12 x_0 x \,(\sd y)^2 . 
\end{align}
Combining \eqref{MDI-mod-02} and \eqref{MDI-mod-03} we get \eqref{MDI-mod-01} for the hard case $\sd x\ge 3 x_0$, $-\sd y \ge y_0/2$. \hfill \qed

\subsection{Proof of Lemma \ref{l:bell-01}}
Proof of this lemma is easier then the proof of Lemma \ref{l:bell-02}. 
Again the statement \cond1 is trivial, we only need to prove \cond2. 

Denote $\cB^{[1]}(X) = (xy)^{1/2}$, $\cB^{[2]}(X) = - xy$, so $\cB= 4 Q^{1/2} \cB^{[1]} + \cB^{[2]}$. As we discusses above, see \eqref{d2B_1}, the function $\cB^{[1]}$ is concave for all $x, y>0$. 

Consider first the case $\sd x\sd y \ge 0$, where, recall $\sd x =(x-x_0)$, $\sd y=y-y_0$. For $t\in[0,1]$ we can write
\begin{align*}
(x_0 + t\sd x)(y_0 + t\sd y) = x_0y_0 +  (\sd x + \sd y) t  +  \sd x \sd y t^2, 
\end{align*}
so subtracting the linear in $t$ term $ (\sd x + \sd y) t$ and substituting $t=1$ we get 
\begin{align}
\label{DI-B_2-01}
 \cB^{[2]}\ci{X_0} (X_0) - \cB^{[2]}\ci{X_0} (X) = \sd x \sd y .
\end{align}
Concavity of $\cB^{[1]}$ means that its contribution is non-negative, so for the case $\sd x\sd y\ge 0$ statement \cond2 of the lemma is proved with $c=1$. 

Let now $\sd x\sd y<0$. We will prove that for any triple $X_1, X_2, X_0 \in (0,\infty)\times(0,\infty)$ such that \[
X_0=\theta_1 X_1 + \theta_2 X_2 =X_0, \qquad \theta_{1,2}\ge 0, \quad \theta_1+\theta_2=1
\]
and such that $X_0\in \Omega_Q$ (i.e.~$x_0y_0\le Q$) and $(x_1-x_0)(y_1-y_0)<0$ we have
\begin{align}
\label{DI-B_2-mod-02}
\cB(X_0) - (\theta_1\cB(X_1) +\theta_2\cB(X_2)) \ge c  \theta_1\theta_2 |x_1-x_2|\cdot |y_1-y_2|\,.
\end{align}
If this inequality is proved, we then take small $h>0$ and define 
\begin{align*}
X_1:=X=X_0+\sd X, \quad X_2 = X_0-h\sd X, \qquad \theta_1= \frac{h}{1+h},  \quad \theta_2= \frac{1}{1+h}. 
\end{align*}
Substituting these values in \eqref{DI-B_2-mod-02}, dividing by $h$ and taking limit as $h\to 0+$ we  will get statement \cond2 of the lemma; here we used the fact that 
\[
\lim_{h\to 0+} \frac1h \left( \cB(X_0) - \frac{1}{1+h} \cB(X_0 - h\sd X) \right) = 
 \cB'(X_0)\sd X. 
\] 

To prove \eqref{DI-B_2-mod-02} we will use concavity of the function $t\mapsto t^{1/2}$. 
Denote 
\[
\sd_x = x_1-x_2 = \sd x/\theta_2, \qquad \sd_y =y_1-y_2 =\sd y/\theta_2. 
\]
Define $h=\theta_1/\theta_2$, so $\theta_1=h/(1+h)$, $\theta_2 = 1/(1+h)$. 
Then combining \eqref{DI-B_2-01} and Lemma \ref{l:conc-triv} we get
\begin{align}
\notag
 \theta_1 x_1 y_1 +\theta_2 x_2y_2 - x_0y_0  & = \left(\theta_2 h^2 + \theta_1\right) \sd x\sd y
\\
\label{DI-B_2-03}
& = h \sd x \sd y =\theta_1\theta_2 \sd_x\sd_y .
\end{align}
Note, that here we have the wrong sign (negative), it wll be compensated by the contribution from $\cB^{[1]}$. Let us estimate that contribution using concavity of the function $t\mapsto t^{1/2}$ and equality  \eqref{DI-B_2-03}: 
\begin{align*}
(x_0y_0)^{1/2} - \theta_1 (x_1y_1)^{1/2} - \theta_2(x_2y_2)^{1/2} &\ge
 (x_0y_0)^{1/2} - \left( \theta_1 x_1 y_1 + \theta_2 x_2 y_2 \right)^{1/2} \\
& =  (x_0y_0)^{1/2} - \left( x_0y_0 + \theta_1\theta_2 \sd_x\sd_y \right)^{1/2}\\
&= \frac{-\theta_1\theta_2 \sd_x\sd_y}{ (x_0y_0)^{1/2}  + \left( x_0y_0 + \theta_1\theta_2 \sd_x\sd_y \right)^{1/2}} \\
& \ge \frac{-\theta_1\theta_2 \sd_x\sd_y}{ 2 (x_0y_0)^{1/2} } \,.
\end{align*}
Multiplying this estimate by $4Q^{1/2}$ and using $x_0y_0\le Q$ we get 
\begin{align*}
4Q^{1/2} \left( \cB^{[1]}(X_0) - (\theta_1\cB^{[1]}(X_1) +\theta_2\cB^{[1]}(X_2) \right) \ge 2  \theta_1\theta_2 |x_1-x_2|\cdot |y_1-y_2|\,.
\end{align*}
Combining this inequality with \eqref{DI-B_2-03} we get \eqref{DI-B_2-mod-02} with $c=1$. \hfill\qed



\section{A counterexample}
\label{counterE}

Here we present a simple  example of a (bounded in a non-weighted $L^2$) martingale transform $T$ and an dyadic $A_2$-weight $w$, such that $T$ is not bounded in $L^2(w)$. 

Take a small $\e>0$. Consider an interval $I$, $|I|=2$, and split it into 4 subintervals (children) $I_k$, $|I_1| = |I_3|=1-\e$, $|I_2|=|I_4|=\e$. Denote $J_1=I_1\cup I_2$, $J_2=I_3\cup I_4$, and define 
\[
h_1:= 2^{-1/2}(\1\ci{J_1} -\1\ci{J_2}), \qquad h_2 := \e^{-1/2}\1\ci{I_2} - \e^{1/2}(1-\e)^{-1} \1\ci{I_1}.  
\]
The functions $h_{1,2}$ are Haar functions, i.e.~they are constant on children of $I$ and orthogonal to constants. Note also that 
\[
\|h_1\|\ci{L^2} =1, \qquad \|h_2\|\ci{L^2} \le 2^{1/2}
\]
(if $\e<\1/2$). Then the operator $T$
\[
T f = (f, h_1)\ci{L^2} h_2 
\]
is a bounded operator in $L^2$, $\|T\|\le 2^{1/2}$. 

Define a weight $w$, 
\[
w(x) := \left\{
\begin{array}{ll} 1, \qquad & x \in  I_1\cup I_3 \,, \\
\e^{-1} &  x\in I_2\cup I_4 \,. \end{array} 
\right. 
\]
Then $w$ satisfies the $A_2$ condition and $[w]\ci{A_2}\le 2$. Here in the definition of $A_2$ condition  we checked the averages over $I$ and over its children $I_k$. Note, that if we also check the $A_2$ condition on intervals $J_{1,2}$, we still have the same estimate $[w]\ci{A_2}\le 2$. But even  if we consider averages over all possible unions of intervals $I_k$, we still have the estimate $[w]\ci{A_2}\le 3$. 

Since $Th_1=h_2$ and 
\[
\| h_1\|\ci{L^2(w)}\le 2, \qquad \|h_2\|\ci{L^2(w)} \ge \e^{-1/2}
\]
we get that 
\[
\|T\|\ci{L^2(w)\to L^2(w)} \ge \e^{-1/2}/2\,.
\]

Considering a sequence of $\e_n\searrow 0$ and taking a direct sum of the above examples, we get a bounded martingale transform $T$ and an $A_2$ weight $w$ such that $T$ is not bounded in $L^2(w)$

\begin{rem*}
A reader familiar with the subject can notice that the operator $T$ in the above counterexample is essentially the Haar shift considered by S.~Petermichl \cite{P}. ``Essentially" means here that we can represent it as a Haar shift on a standard dyadic lattice as in \cite{P}, but we have to change the reference measure from Lebesgue measure to a certain very non-doubling measure $\nu$.
\end{rem*}

\section{\texorpdfstring{$T(\1)$}{T(1)} theorem for Haar multipliers} 
\label{T1}

In this section we will prove that it is sufficient to check the weighted estimates for Haar multipliers on characteristic functions on atoms. 

While we will need the result only for absolutely continuous (with respect to the ambient measure $\nu$) measures, we state it here for arbitrary  measures $\mu_1$, $\mu_2$. 

 In what follows we will only consider finite sequences $\sigma=(\sigma\ci I)\ci{I\in\cD}$, $|\sigma\ci I|\le 1$ (only finitely many terms are non-zero), thus avoiding unnecessary technical details.

In this section measuer $\mu$ is also an arbitrary measure.
Note that  for a measure $\mu$ and $f\in L^1(I, \mu)$ we can define 
\[
\E\ci I(f\mu):= \left(|I|^{-1} \int_I f\,d\mu \right) \1\ci I =\La f\mu\Ra\ci I \1\ci I, 
\]
and therefore $\Delta\ci I (f\mu)$. Then for  the martingale multiplier $T=T_\sigma$ we can define $T(f\mu)$. 

Recall that for a martingale multiplier $T=T_\sigma$ and $I_0\in\cD$ we defined the operator $T\ci I$
\[
T\ci{I_0} = \sum_{I\in\cD(I_0)} \sigma\ci{I} \Delta\ci{I}\,.
\]
Note also  that for $f\in L^1(I, \mu)$, $I\in\cD$ the function $T\ci I (f\mu)$ is well defined.


\begin{thm}
\label{t:T1}
Let $T=T_\sigma$, $\sigma=(\sigma\ci I)\ci{I\in\cD}$, $|\sigma\ci I|\le 1$ be a Haar multiplier, and let $\mu_1$, $\mu_2$ be measures on $\cX$ such that 
\[
\sup_{I\in\cD} |I|^{-2}\mu_1(I)\mu_2(I) =: [\mu_1, \mu_2]\ci{A_2} <\infty. 
\] 
Assume that for all $I\in\cD$
\begin{align}
\label{test-07}
\|T\ci I (\1\ci I \mu_1)\|\ci{L^2(\mu_2)} &\le A \|\1\ci I\|\ci{L^2(\mu_1)} = A\mu_1(I)^{1/2}, 
\\
\notag
\|T\ci I (\1\ci I \mu_2)\|\ci{L^2(\mu_1)} &\le A \|\1\ci I\|\ci{L^2(\mu_2)} = A\mu_2(I)^{1/2}
\end{align}
Then 
\[
\|Tf\mu_1\|\ci{L^2(\mu_2)} \le \left(2 [\mu_1, \mu_2]^{1/2}\ci{A_2} +5A\right) \|f\|\ci{L^2(\mu_1)} \,.
\]
\end{thm}

\begin{proof}
For a measure $\mu$ finite on each $I\in\cD$ define the weighted averaging operators $\E\ci I^\mu$, $I\in\cD$
\begin{align*}
\E\ci I^\mu f &:= \left(\mu(I)^{-1}\int_I f\, d\mu \right) \1\ci I =: \La f \Ra\ci{I,\mu} \1\ci I \,, 
\intertext{and the weighted martingale differences $\Delta\ci I^\mu$,} 
\Delta\ci I^\mu & := -  \E\ci I^\mu + \sum_{I'\in\ch(I)} \E\ci{I'}^\mu \,. 
\end{align*}
The subspaces $\Delta\ci I^\mu L^2(\mu)$ are orthogonal in $L^2(\mu)$, the operators $\Delta\ci I^\mu$ are orthogonal projections onto these subspaces, so for all $f\in L^2(\mu)$ 
\[
\sum_{I\in\cD} \|\Delta\ci I^\mu f\|\ci{L^2(\mu)}^2 \le \|f\|\ci{L^2(\mu)}^2 .
\]

Define an operator $T^{\mu_1} : L^2(\mu_1) \to L^2(\mu_2)$ as 
\[
T^{\mu_1} f = T(f\mu_1), \qquad f\in L^2(\mu_1)\,.
\]
Its dual with respect to the linear dualities $\La \fdot, \fdot \Ra_{\mu_{1,2}}$, 
\[
\La f, g \Ra_\mu = \int fg\,d\mu\,,
\]
is the operator $T^{\mu_2} :L^2(\mu_2)\to L^2(\mu_1)$, 
\[
T^{\mu_2} f = T(f\mu_2) , \qquad f \in L^2(\mu_2). 
\]

Define the \emph{paraproducts} $\Pi_1=\Pi\ci{T^{\mu_1}}: L^2(\mu_1) \to L^2(\mu_2)$
\begin{align*}
\Pi_1 f & = \sum_{I\in \cD} (\E\ci I^{\mu_1} f) \Delta\ci I^{\mu_2} (T^{\mu_1}\1\ci I)\,,
\intertext{and $\Pi_2=\Pi\ci{T^{\mu_2}}: L^2(\mu_2) \to L^2(\mu_1)$}
\Pi_2 f & = \sum_{I\in \cD} (\E\ci I^{\mu_2} f) \Delta\ci I^{\mu_1} (T^{\mu_2}\1\ci I)\,.
\end{align*}
Note that if $I\subset I_0$, $I_0\in\cD$, then 
\begin{align}
\label{DtT1}
\Delta\ci I^{\mu_2} T^{\mu_1} \1\ci I = \Delta\ci I^{\mu_2} T^{\mu_{1}} \1\ci{I_0}  
= \Delta\ci I^{\mu_2} T^{\mu_{1}}\ci{I_0} \1\ci{I_0} , 
\end{align}
so we (at least formally) can write $T^{\mu_{1,2}}\1$ instead of $ T^{\mu_{1,2}}\1\ci I$ in the definition of paraproducts. 

If $I'\in\cD(I)$, $I'\ne I$ then 
\begin{align*}
\Delta\ci{I'}^{\mu_2} T^{\mu_1} \Delta\ci I^{\mu_1} = \Delta\ci{I'}^{\mu_2} T\ci I^{\mu_1} \Delta\ci I^{\mu_1} = \Delta\ci{I'}^{\mu_2} \Pi_1 \Delta\ci I^{\mu_1}\,,
\end{align*} 
and if $I'\in \cD$ does not intersect $I$ then 
\begin{align*}
\Delta\ci{I'}^{\mu_2} T^{\mu_1} \Delta\ci I^{\mu_1} 
= \Delta\ci{I'}^{\mu_2} \Pi_1 \Delta\ci I^{\mu_1} =0 \,.
\end{align*} 
Finally, if $I'\in\cD(I)$, then $\Delta\ci{I}^{\mu_2} \Pi_1 \Delta\ci{I'}^{\mu_1}=0$. 

Similar formulas hold for $T^{\mu_2}$ and $\Pi_2$, so we can represent 
\[
T^{\mu_1} = \Pi_1 + \Pi_2' + T^{\mu_1}\ti{diag}\,,
\]
where $\Pi_2'$ is the dual (with respect to the linear duality) of the paraproduct $\Pi_1$, and $T^{\mu_1}\ti{diag} $ is a ``diagonal'' operator, meaning that
\[
T^{\mu_1}\ti{diag} f = \sum_{I\in\cD} \Delta\ci{I}^{\mu_2} T^{\mu_1} \Delta\ci I^{\mu_1} f.  
\]

For the paraproduct $\Pi_1$ we have
\[
\|\Pi_1 f\|^2\ci{L^2(\mu_2)} = \sum_{I\in\cD}|\La f\Ra\ci{I,\mu_1} |^2 
\| \Delta\ci I^{\mu_2} T^{\mu_1} \1\ci I\|\ci{L^2(\mu_2)}^2, 
\]
so we can estimate its norm using  the Carleson Martingale Embedding Theorem. 
We get using \eqref{DtT1} that  for any $I_0\in\cD$
\begin{align*}
\sum_{I\in\cD(I_0)} 
\|\Delta\ci I^{\mu_2} T^{\mu_1} \1\ci I\|\ci{L^2(\mu_2)}^2 
& = \sum_{I\in\cD(I_0)} 
\|\Delta\ci I^{\mu_2} T\ci{I_0}^{\mu_1} \1\ci{I_0}\|\ci{L^2(\mu_2)}^2
\\
& \le  \| T\ci{I_0}^{\mu_1} \1\ci{I_0}\|\ci{L^2(\mu_2)}^2 \le A^2 \mu(I_0)
\end{align*}
and by the Carleson Martingale Embedding Theorem $\|\Pi_1\|\le 2 A$. Similarly, $\|\Pi_2\|\le 2A$, so it remains to estimate $T^{\mu_1}\ti{diag}$. 

Since $\Delta\ci I^{\mu_2}\Pi_1 \Delta\ci I^{\mu_1} =0$ and similarly for $\Pi_2$, we get that
\[
\Delta\ci I^{\mu_2}T^{\mu_1}\ti{diag} \Delta\ci I^{\mu_1} = 
\Delta\ci I^{\mu_2}T^{\mu_1}  \Delta\ci I^{\mu_1}\,.
\]
Moreover, since for $J\notin\cD(I)$
\begin{align*}
\Delta\ci J [\Delta\ci{I}^{\mu_1} (f) \mu_1]= 0, 
\intertext{we conclude that}
\Delta\ci I^{\mu_2}T^{\mu_1}\ti{diag} \Delta\ci I^{\mu_1} = 
\Delta\ci I^{\mu_2}T^{\mu_1}\ci I  \Delta\ci I^{\mu_1}\,.
\end{align*}

Denote 
\[
\Delta\ci I^{\mu_1} f =: h = \sum_{I'\in \ch(I)} \alpha\ci{I'}\1\ci{I'}\,. 
\]
Since $T^{\mu_1}\ci I f= \sigma\ci I \Delta\ci I(f\mu_1) +  \sum_{I'\in\ch(I)} T^{\mu_1}\ci{I'}f$, 
\[
T^{\mu_1}\ci I h  = \sigma\ci I \Delta\ci I (h\mu_1) + 
\sum_{I'\in\ch(I)} \alpha\ci{I'} T^{\mu_1}\ci{I'}\1\ci{I'} =: g_1 + g_2.
\]
Using \eqref{test-03} we get that
\begin{align*}
\| g_2\|\ci{L^2(\mu_2)}^2 & = \sum_{I'\in\ch(I)} |\alpha\ci{I'}|^2 
\| T^{\mu_1}\ci{I'}\1\ci{I'} \|\ci{L^2(\mu_2)}^2 
\\
& \le
A^2\sum_{I'\in\ch(I)} |\alpha\ci{I'}|^2 
\| \1\ci{I'} \|\ci{L^2(\mu_1)}^2 = A^2 \|h\|\ci{L^2(\mu_1)}^2.
\end{align*}
Recalling the definition \eqref{Delta_I} of $\Delta\ci I$ and the fact that the norm of the averaging operator $f \mapsto \E\ci I (f\mu_1)$ as an operator $L^2(\mu_1)\to L^2(\mu_2)$ is exactly $|I|^{-1} \mu_1(I)^{1/2}\mu_2(I)^{1/2}$, we get that 
\[
\|g_1\|\ci{L^2(\mu_2)} \le 2 [\mu_1, \mu_2]\ci{A_2}^{1/2} \|h\|\ci{L^2(\mu_1)}\,.
\]
(Finding the norm of the averaging operator is an easy computation that we leave as an exercise: anybody should be able to compute a norm of a rank one operator.)

Since an operator $\Delta^{\mu_2}\ci I$ is an orthogonal projection in $L^2(\mu_2)$ we conclude that the norm of each diagonal block $\Delta\ci I^{\mu_2}T^{\mu_1}\ti{diag} \Delta\ci I^{\mu_1}$, and so the norm of the operator $T^{\mu_1}\ti{diag}$ can be estimated as 
\[
\| T^{\mu_1}\ti{diag} \|\ci{L^2(\mu_1)\to L^2(\mu_2)} \le A + 2 [\mu_1, \mu_2]\ci{A_2}^{1/2}\,.
\]
Combining that with estimates of paraproducts we get the conclusion of the theorem. 
\end{proof}

\section{Reduction to Carleson measures without \texorpdfstring{$T(1)$}{T(1)} theorem}
\label{withoutT1}

In some cases the mechanism of two weight $T(1)$ theorem can be rather involved, we were lucky that for well localized operators the mechanism is more simple.
So sometimes it is convenient to see  the direct proof, without using the mechanism of  a two-weight $T(1)$ theorem.  Also getting a proof that circumvents this mechanism can be instructive. This second proof will require a bit more of outer measure spaces machinery (but not much more). 

 The main estimate  we want to prove is \eqref{main2w-01}.
By duality,  it follows from the
inequality ($u:=w^{-1}$):
\begin{equation}
\label{bilin1}
\sum_{I\in \cD} \bigg|\int_I \Delta\ci I(fu) \cdot \Delta\ci I(gw) \, d\nu\bigg| \le A
[w]\ci{A_2} \|f\|\ci{L^2(u)} \|g\|\ci{L^2(w)}\ .
\end{equation}

Recalling the definition of $\E\ci I^\mu$ and $\Delta\ci I^\mu$ from Section \ref{T1} and 
using the notation $\E\ci I^u$ and $\Delta\ci I^u$ for $d\mu= ud\nu$, 
let us split this sum to $4$ sums. 
For that we will write first
\begin{align*}
\Delta\ci I (fu) &= \Delta\ci I[(f -\E\ci I^u f -\Delta\ci I^u f)u] + \Delta\ci I[ \E\ci I^u f\cdot u] +\Delta\ci I[\Delta\ci I^u f\cdot u]\ .
\\
\Delta\ci I (gw) &= \Delta\ci I[(g -\E\ci I^w g -\Delta\ci I^w g)w] + \Delta\ci I[ \E\ci I^w g\cdot w] +\Delta\ci I[\Delta\ci I^w g\cdot w]\ .
\end{align*}

Now let us notice that the first terms in these formulae vanish. In fact, $1_I(f -\E\ci I^u f -\Delta\ci I^u f)$ is the sum of ``small" intervals Haar functions, meaning that this term (on $I$) is the sum of $\Delta\ci J^u f$, where $J\subsetneq I$. But for every such $J$, $\Delta\ci J^u f$ is ortogonal to constants in $L^2(u)$. Operation $\Delta\ci I (u\Delta\ci J^u f)$ thus returns $0$ because such $J$ is inevitably equal to or is a descendant of $I', I'\in \ch(I)$.

As $\E\ci I^u f$ is constant on $I$, we have $\Delta\ci I[ \E\ci I^u f\cdot u] = (\E\ci I^u f   ) \Delta\ci I u$ and similarly for $w$, so one can now estimate \eqref{bilin1} by the  $4$ sums:
\begin{align*}
\Sigma_4 &= \sum_{I\in \cD} \bigg|\int_I (\E\ci I^u f) (\E\ci I^w g) \cdot(\Delta\ci I u)\,\,( \Delta\ci I w)\, d\nu\bigg|\ , 
\\
\Sigma_3 & = \sum_{I\in \cD} \bigg|\int_I  \Delta\ci I[(\Delta\ci I^u f) u]\cdot(\E\ci I^w g)\,\Delta\ci I w\, d\nu\bigg|\ ,
\\
\Sigma_2 & = \sum_{I\in \cD} \bigg|\int_I  \Delta\ci I[(\Delta\ci I^w g) w]\cdot(\E\ci I^u f)\,\Delta\ci I u\, d\nu\bigg|\ ,
\\
\Sigma_1 &= \sum_{I\in \cD} \bigg|\int_I  \Delta\ci I[(\Delta\ci I^u f) u]\cdot \Delta\ci I[(\Delta\ci I^w g) w]\, d\nu\bigg|\ .
\end{align*}

Let us look at $\Sigma_4$ first: since
\[
\Bigl|\int_I \Delta\ci I u\,\Delta\ci I w\,d\nu\Bigr| \le \rho_I |I|,
\] 
where as in Section \ref{s:FirstReduct}
\[
\rho\ci I:= \sum_{I'\in\ch(I)}  | \La u\Ra\ci{I'} - \La u\Ra\ci{I} |\cdot 
| \La w\Ra\ci{I'} - \La w\Ra\ci{I} |\cdot|I'|/|I|, 
\]
we can estimate
\[
\Sigma_4 \le \sum_{I\in\cD}\frac{\left| \La fu \Ra\ci I \right|}{\La u \Ra\ci I} \cdot
\frac{\left| \La gw \Ra\ci I \right|}{\La w \Ra\ci I} \cdot\rho\ci I |I|;
\]
here as usual we use the notation $|I|=\nu(I)$. 

And the last sum was, in essence, already estimated in Section \ref{averagin&OMS}. Namely, it has been proved there that the collection $\{\rho_I \}_{I\in \cD}$ has the Carleson packing property  \eqref{CMP-02}. In particular, this means that
\begin{equation}
\label{rhorho}
\{\rho_I\}\ci{I\in \cD} \in L^\infty(\cD, S^1), \qquad
\big\| \{\rho_I\}\ci{I\in \cD} \big\|\ci{L^\infty(\cD, S^1)} \le C [w]\ci{A_2}  \ .
\end{equation}
But by Lemma \ref{l:bilin}
\begin{align}
\notag
\left\{ \Phi\ci I\right\}\ci{I\in\cD}  := 
\left\{ \frac{\left| \La fu \Ra\ci I \right|}{\La u \Ra\ci I} \cdot
\frac{\left| \La gw \Ra\ci I \right|}{\La w \Ra\ci I} \right\}_{I\in \cD} &\in L^1(\cD, S^\infty)\,, 
\\ 
\label{fg}
\left\|  \left\{ \Phi\ci I\right\}\ci{I\in\cD} \right\|\ci{L^1(\cD, S^\infty)} & \le 4 \| f\|\ci{L^2(u)} \|g\|\ci{L^2(w)}.
\end{align}
Now we use $L^1$--$L^\infty$ duality (Lemma \ref{l:L1-L^infty}) with $\mu=\nu$ to conclude 
that
\[
\Sigma_4
\le C[w]\ci{A_2}\|f\|\ci{L^2(u)} \|g\|\ci{L^2(w)}\ .
\]

Sums $\Sigma_2$ and $\Sigma_3$ are symmetric, so let us consider $\Sigma_2$.
\begin{align*}
\Sigma_2 &=\bigg|\int_I \Delta\ci I[(\Delta\ci I^w g) w]\cdot\E\ci I^u(f)\,\Delta\ci I u\, d\nu\bigg|=
\bigg|\int_I (\Delta\ci I^w g) w\cdot \E\ci I^u(f)\Delta\ci I u)\, d\nu\bigg|
\\
 &\le \bigg(\sum_I \int_I|\Delta\ci I^w g|^2 w\, d\nu\bigg)^{1/2} \bigg(\sum_I\int_I|\E\ci I^u f|^2|\Delta\ci Iu|^2 w\, d\nu\bigg)^{1/2}
\\
&\le\|g\|\ci{L^2(w)}\bigg(\sum_I\bigg|\frac{\La fu\Ra\ci I}{\La u\Ra\ci I}\bigg|^2\int_I|\Delta\ci Iu|^2 w\, d\nu\bigg)^{1/2}
\\
& =
\|g\|\ci{L^2(w)}\bigg(\sum_I\big|\La f\Ra\ci{I, u}\big|^2\int_I|\Delta\ci I u|^2 w\, d\nu\bigg)^{1/2}
=
\|g\|\ci{L^2(w)}\bigg(\sum_I\big|\La f\Ra\ci{I, u}\big|^2\gamma\ci I |I|\bigg)^{1/2}\ , \end{align*}
where
\[
\gamma\ci I:=|I|^{-1}\int_I|\Delta\ci I u|^2 w\, d\nu = \sum_{I'\in\ch(I)} | \La u\Ra\ci{I'} - \La u\Ra\ci{I} |^2 \ci{I}\La w\Ra\ci{I'} |I'|/|I| \,.
\]
So, to get the correct estimate of $\Sigma_2$ it is sufficient to show that 
\begin{align}
\label{embed-03}
\sum_I\big|\La f\Ra\ci{I, u}\big|^2\gamma\ci I |I| \le C_1 [w]\ci{A_2}^2 \|f\|\ci{L^2(u)}^2 \qquad \forall f \in L^2(u).
\end{align}
By the martingale Carleson Embedding Theorem it is sufficient to test the embedding on characteristic functions $\1\ci{I_0}$, $I_0\in\cD$, so \eqref{embed-03} follows from the estimate 
\begin{align}
\label{test-08}
\sum_{I\in\cD(I_0)} \gamma\ci I |I| \le C [w]\ci{A_2}^2 \|\1\ci{I_0}\|\ci{L^2(u)}^2 =
C [w]\ci{A_2}^2 \La u \Ra\ci{I_0} |I_0| \qquad \forall I_0\in\cD, 
\end{align}
and this estimate implies \eqref{embed-03} with  $C_1=4C$. 

But the above estimate \eqref{test-08} is exactly the estimate \eqref{test-05} which was proved before in Section \ref{s:ProofMainRes}.

The martingale Carleson Embedding Theorem, i.e.~the fact that \eqref{test-08} implies \eqref{embed-03} can be also shown using the machinery of the outer measure spaces.


To see that we prove the following lemma,  which essentially encodes the $L^2$ 
boundedness of the martingale 
maximal operator. 

Recall that for a finite on all $I\in\cD$ measure $\mu$ on $\cX$, the averaging operator $\A_\mu$, mapping functions on $\cX$ to functions on $\cD$ is defined as 
\[
\A_\mu f (I) := \La f\Ra\ci{I, \mu} := \mu(I)^{-1} \int_I f d\mu , \qquad I\in \cD; 
\]
if $\mu(I)=0$ we put $\A_\mu f(I) =0$. 
\begin{lm}
\label{l2}
The averaging operator $\A_\mu$ is a bounded operator from $L^2(\mu)$ to 
the outer space $L^2(\cD, \mu^*, S^\infty)$ and, moreover
\[
\| \A_\mu f\|\ci{ L^2(\cD, \mu^*, S^\infty)} \le 2 \|f\|\ci{L^2(\mu)} \qquad \forall f\in L^2(\mu).
\]
\end{lm}

Applying this lemma with $d\mu= u d\nu$ we get that $\left\{ \La f\Ra\ci{I,\mu}\right\}\ci{I\in\cD} \in L^2(\cD, \mu^*, \, S^\infty_\mu)$, or, equivalently $\big\{ \La f\Ra\ci{I,\mu}^2\big\}\ci{I\in\cD} \in L^1(\cD, \mu^*, \, S^\infty)$ and 
\[
\Big\| \big\{ \La f\Ra\ci{I,\mu}^2\big\}\ci{I\in\cD}  \Big\|\ci{ L^1(\cD, \mu^*, \, S^\infty) } \le 4 \|f\|\ci{L^2(\mu)}^2. 
\]
On the other hand, \eqref{test-08} means that $\{\gamma\ci I/\La u\Ra\ci I \}\ci{I\in\cD} \in L^\infty(\cD, \mu^*, \, S^1_u)$, $d\mu=u d\nu$, 
\[
\big\| \{\gamma\ci I /\La u\Ra\ci I \}\ci{I\in\cD} \big\|\ci{ L^\infty(\cD, \mu^*, \,S^1_u) } 
\le C [w]\ci{A_2}^2 \,.
\]
Applying Lemma \ref{l:L1-L^infty}   ($L^1$--$L^\infty$ duality) we get \eqref{embed-03} with $C_1=4C$. 

\begin{proof}[Proof of Lemma \ref{l2}]
It is sufficient to prove lemma for functions $f$ supported on a union of finitely many intervals $I\in\cD$. It is also sufficient to consider only $f\ge 0$. 

Fix $\lambda>0$ and consider the maximal intervals $I\in \cD$ such that
\[
F(I):= \A_\mu f(I)>\lambda;
\]
since $f$ is supported on a union of finitely many intervals, such maximal $I$ always exist.

Call the family of such intervals $\cH_\lambda$. 
Since $F(I) = \La f \Ra\ci{I,\mu}$ 
we observe that 
for any $I\in \cH_\lambda$ and $x\in I$
\[
\lambda< F(I)\le (M_\mu\ut d f)(x)\ , 
\]
where $M_\mu\ut d$ it the maximal operator defined by \eqref{MaxOp}. 

Hence
\[
\mu^* \Big\{\bigcup_{I\in \cH_\lambda} \cD(I)\Big\} \le \mu \{x: M_\mu\ut d f(x)>\lambda\}.
\]
If we consider the new function $K$ on $\cD(I)$,  which is equal to $F(I)$ on 
$\cD\setminus \bigcup_{J\in \cH_\lambda} \cD(J)$ and zero in $\bigcup_{J\in \cH_\lambda} \cD(J)$, 
we readily see that $S^\infty_\mu (K, \cD(I)) \le \lambda$ for all $I\in \cD$. 
Therefore, by the definition 
of the outer measure of the super level set  we have
\[
\mu^*(  S^\infty_\mu F >\lambda) \le \mu\{\bigcup_{J\in \cH_\lambda} J\} \le \mu\{x: M_\mu\ut d f(x)>\lambda\}.
\]
Now multiplying both sides by $p\lambda^{p-1}$ and integrating with respect to $d\lambda$ we get
\[
\|F\|\ci{L^2(\cD(I),\mu^*, S^\infty_\mu)}^2
\le \|M\ut d_\mu h\|\ci{L^2(\mu)}^2 \le 4\|h\|\ci{L^2(\mu)}^2\,,
\]
which proves the lemma. 
\end{proof}

Finally, let us estimate the sum $\Sigma_1$:
\begin{align*}
\Sigma_1 & = \sum_{I\in \cD} \bigg|\int_I  \Delta\ci I[(\Delta\ci I^u f) u]\cdot \Delta\ci I[(\Delta\ci I^w g) w]\, d\nu\bigg| 
\\
&=
\sum_{I\in \cD} \bigg|
\int_I  (-\E\ci I[(\Delta\ci I^u f) u] +\sum_{I'\in \ch(I)}\E\ci {I'}[(\Delta\ci I^u f) u])\cdot [(\Delta\ci I^w g) w]\, d\nu\bigg|
\\
&= \sum_{I\in\cD} \Bigg|\int_I  \bigg(\sum_{I'\in \ch(I)}\E\ci{I'}[(\Delta\ci I^u f) u]\bigg)\,(\Delta\ci I^w g) \cdot w\, d\nu\Bigg| 
\\
&\le \sum_{I\in\cD} \Bigg(\int_I  \bigg|\sum_{I'\in \ch(I)}\E\ci {I'}[(\Delta\ci I^u f) u]\bigg|^2 w d\nu\Bigg)^{1/2}\bigg(\int_I\big| \Delta\ci I^w g\big|^2 \cdot w\, d\nu\bigg)^{1/2}
\end{align*}
Since  the intervals $I'\in\cD(I)$ are disjoint, we can estimate
\begin{align*}
\int_I  \bigg|\sum_{I'\in \ch(I)}\E\ci {I'}[(\Delta\ci I^u f) u]\bigg|^2 w d\nu &=\sum_{I'\in \ch(I)}\int_{I'}  (\E\ci {I'}[(\Delta\ci I^u f) u])^2 w d\nu 
\\
&= \sum_{I'\in \ch(I)}\La \Delta\ci I^u f\Ra\ci{I', u}^2\La u\Ra\ci{I'}^2 \La w\Ra\ci{I'} |I'| 
\\
 & \le [w]\ci{A_2}\,  \sum_{I'\in \ch(I)}\La \Delta\ci I^u f\Ra\ci{I', u}^2\,u(I') =[w]\ci{A_2}\, \|\Delta^u\ci I f\|\ci{L^2(u)}^2. 
\end{align*}
Hence,
\begin{align*}
\Sigma_1  & \le [w]\ci{A_2}^{1/2}\sum_{I\in\cD} \|\Delta\ci I^w g\|\ci{L^2(w)} 
\| \Delta\ci I^u f \|\ci{L^2(u)}
\\
&\le [w]\ci{A_2}^{1/2} \bigg(\sum_{I\in\cD} \|\Delta\ci I^w g\|\ci{L^2(w)}^2 \bigg)^{1/2} \bigg( \sum_{I\in\cD} \|\Delta\ci I^u f \|\ci{L^2(u)}^2 \bigg)^{1/2} 
\\
& \le [w]\ci{A_2}^{1/2} \|f\|\ci{L^2(u)}\|g\|\ci{L^2(w)}\ .
\end{align*}


\begin{thebibliography}{99}

\bibitem{DT} {\sc Y. Do, C. Thiele},   
$L^p$ theory for outer measures and two themes of Lennart Carleson united, to appear
in Bulletin AMS, arXiv:1309.0945.

\bibitem{DGPP} {\sc  O.  Dragicevic,  L. Grafakos, M. C. Pereyra,  S. Petermichl}, Extrapolation and sharp norm estimates for classical operators on weighted Lebesgue spaces. Publ. Mat. 49 (2005), no. 1, 73--91. 

\bibitem{GR}{\sc Jos\'e Garcia-Cuerva, Jos\'e L. Rubio de Francia}, Weighted Norm Inequalities and Related Topics, North-Holland, Mathematics Studies, v. 116, Amsterdam, New York, Oxford, 1985.


\bibitem{GundyWheeden1973}
{\sc R.~F. Gundy and R.~L. Wheeden,} Weighted integral inequalities for the
  nontangential maximal function, {L}usin area integral, and {W}alsh-{P}aley
  series, Studia Math. \textbf{49} (1973/74), 107--124. 
  
  
\bibitem{Hy} {\sc T.  Hyt\"onen,} The sharp weighted bound for general Calder\'on-Zygmund operators. Ann. of Math. (2) 175 (2012), no. 3, 1473--1506. 

\bibitem{HPTV} {\sc T. Hyt\"onen, C. P\'erez, S. Treil, A. Volberg}, Sharp weighted estimates for dyadic shifts and the $A_2$ conjecture. J. Reine Angew. Math. 687 (2014), 43--86. 

\bibitem{Krantzberg-1971} {\sc A.~S.~Krantzberg,} On the basicity of the Haar system in the weighted spaces, Mosk. Inst. Electr.Mat, \textbf{24} (1971). 

\bibitem{Le} {\sc A. Lerner}, A simple proof of $A_2$ conjecture, Intern. Math. Res. Notices {\bf IMRN}, 2013, no. 14, 3159--3170. 

\bibitem{LMP} {\sc L.~D.~L\'opez-S\'anchez, J.-M.~Martell, J.~Parcet}, Dyadic harmonic analysis beyond doubling measures, arXiv:1211.6291, to appear in Advances in Math.

\bibitem{NTV} {\sc F. Nazarov, S. Treil, A. Volberg}, The Bellman functions and two-weight inequalities for Haar multipliers, J. Amer. Math. Soc. 12 (1999), no. 4, 909--928.  

\bibitem{NTVmrl} {\sc F. Nazarov, S. Treil, A. Volberg,} Two weight inequalities for individual Haar multipliers and other well localized operators. Math. Res. Lett. 15 (2008), no. 3, 583--597.

\bibitem{NV}  {\sc F. Nazarov, A. Volberg}, A simple sharp weighted estimate of the dyadic shifts on metric space with geometric doubling, arxiv:1104.4893,  Int. Math. Res. Not. {\bf IMRN} 2013, no. 16, 3771--3789. 

\bibitem{PV} {\sc S. Petermichl, A. Volberg,}  Heating of the Ahlfors-Beurling operator: weakly quasiregular maps on the plane are quasiregular. Duke Math. J. 112 (2002), no. 2, 281--305.

\bibitem{P} {\sc  S. Petermichl, } The sharp bound for the Hilbert transform on weighted Lebesgue spaces in terms of the classical $A_p$ characteristic. Amer. J. Math. 129 (2007), no. 5, 1355--1375. 

\bibitem{Tr_Comm-para2010}
{\sc S.~Treil}, Commutators, paraproducts and BMO in non-homogeneous martingale
  settings, Rev. Mat. Iberoam. \textbf{29} (2013), no.~4,
  1325--1372, see also arXiv:1007.1210 [math.CA] (2010).
  
  
\bibitem{TV} {\sc S. Treil, A. Volberg}, Entropy conditions in two weight inequalities for singular integral 
operators, arxiv:1405.0385.

\bibitem{VV}
{\sc V. Vasyunin, A. Volberg}, The  Bellman functions for the simlplest two-weight inequality: an investigation of a particular case, St. Petersburg Math. J. 18 (2007), no. 2, 201--222.

\bibitem{Wi} {\sc J. Wittwer}, A sharp estimate on the norm of martingale transform, Math. Res. Letters, v. 7 (2000), pp. 1--12.

\end{thebibliography}
\end{document}